\documentclass[final,3p,times]{elsarticle}
\usepackage{amssymb}
\usepackage{amsthm}
\usepackage{latexsym}
\usepackage{amsmath}
\usepackage{amsmath}
\usepackage{color}
\usepackage{graphicx}
\usepackage{indentfirst}
\usepackage{mathrsfs}
\usepackage{pdfsync}
\usepackage{bm}
\usepackage{hyperref}
\hypersetup{hypertex=true,
	colorlinks=true,
	linkcolor=blue,
	anchorcolor=blue,
	citecolor=blue}
\numberwithin{equation}{section}
\allowdisplaybreaks
\usepackage[utf8]{inputenc}

\newtheorem{theorem}{\color{black}\indent Theorem}[section]

\newtheorem{lemma}{\color{black}\indent Lemma}[section]

\newtheorem{definition}{\color{black}\indent Definition}[section]
\newtheorem{remark}{\color{black}\indent Remark}[section]

\journal{ }

\begin{document}

\begin{frontmatter}

    \title{Statistical ensembles in integrable Hamiltonian systems with almost periodic transitions}

    \author[author1]{Xinyu Liu \footnote{ E-mail address : liuxy595@nenu.edu.cn}}
    \author[author2,author3]{Yong Li \footnote{ E-mail address : liyong@jlu.edu.cn} }
    \address[author1]{School of Mathematics and Statistics, Northeast Nolmal University, Changchun 130024, P. R. China.}
    \address[author2]{College of Mathematics, Jilin University, Changchun 130012, P. R. China.}
    \address[author3]{School of Mathematics and Statistics, Center for Mathematics and Interdisciplinary Sciences, Northeast Normal University, Changchun 130024, P. R. China.}

    \cortext[cor1]{Corresponding author at : College of Mathematics, Jilin University,Changchun 130012, P. R. China.} 

    \begin{abstract}
     We study the long-term average evolution of the random ensemble along integrable Hamiltonian systems with time $T$-periodic transitions. More precisely, for any observable $G$, it is demonstrated that the ensemble under $G$ in long time average converges to that over one time period $T$, and that the probability measure induced by the probability density function describing the ensemble at time $t$ weakly converges to the average of the probability measures over time $T$. And we extend the result to almost periodic cases. The key to the proof is based on the {\it {Riemann-Lebesgue lemma in time-average form}} generalized in the paper. 

    \end{abstract}


    \end{frontmatter}

    \section{Introduction}\label{sec:1}
    \setcounter{equation}{0}
    \setcounter{definition}{0}
    \setcounter{proposition}{0}
    \setcounter{lemma}{0}
    \setcounter{remark}{0}

    The statistical ensemble (also known as the Gibbs ensemble; for a comprehensive physical understanding of the ensemble, see  \cite{penrose1979foundations}) plays an essential role in studying statistical physics. 
Mathematically, it frequently is reduced to how evolving the distribution of initial random variables along some Hamiltonian systems over time. 
Joel L Lebowitz and Oliver Penrose \cite{lebowitz1973modern,walters2000introduction} provided an ergodic condition for the relaxation of the orbital ensemble of fixed energy $E$ to a microcanonical ensemble of equilibrium statistical mechanics. Emil A Yuzbashyan \cite{yuzbashyan2016generalized}
established such the long-term ensemble for Hamiltonian systems including quantum ones. In recent years, mathematicians and physicists have made significant strides in the study of ensembles. In \cite{goldfriend2019equilibration}, Goldfriend Tomer and Kurchan Jorge focused on the classical Fermi-Pata-Ulam-Tsingou (FPU) chain problem in their investigation of the slow relaxation problem of isolated quasi-integrable systems. Through numerical simulations, they demonstrated how the slow drift in Toda motion integral space governs the relaxation of the FPU chain in the direction of equilibrium. Damien Barbier et al. demonstrated that the Generalized Gibbs Ensemble produced by N conserved charges in the involution captures the long-time average of all significant observables of the soft model following parameter mutation \cite{barbier2020non}.
In \cite{matsui2020nonequilibrium}, Chihiro Matsui demonstrated that the Generalized Gibbs Ensemble of the gapless XXZ model is not linearly independent, but rather consists of conserved values that are functionally independent. Also revealed was the physical significance of spin-flip non-invariant conservation. 
    In \cite{mitchell2019weak}, Chad Mitchell showed the long-term behavior of bounded orbits associated with the initial ensemble in a nondegenerate integrable Hamiltonian system, and provided the requirements for the ensemble's convergence in the form of action angle to the equilibrium state.

    Notably, the differential equations explored in the preceding work are all idealized, i.e., they are continuous and differentiable at all times.

    In this paper, differ from the preceding ones, we consider the statistical ensemble problem for integrable Hamiltonian systems in which the action variable is modified by periodic transitions and almost periodic transitions. This abrupt transitions have significant physical relevance, as it is widely known that many practical physical processes in nature are frequently accompanied by a short-term disturbance 
    which results in some state transition. And in this regard, there have been some well-known outcomes, such as, in \cite{figalli2016weak}, Figalli et al. addressed the optimal switching problem, in which the switching occurs at a specific period by establishing the weak KAM theory and Aubry Mather theory for the problem. 
    

    Our primary concern is what changes will occur in the ensemble's time-averaged final output if it is abruptly transitions throughout its evolution. It is possible to determine the final direction of the statistical ensemble when systems suffer discontinuous transitions. We can give a positive answer under certain conditions. The asymptotic behavior of ensemble is also demonstrated in this article.

    The main technology of this paper is the {\it {Riemann-Lebesgue lemma in time-average form}}. The technique described in Reference \cite{mitchell2019weak} is no longer relevant when the transitions is periodic because the action variable $I$ suffers abrupt changes over time. This necessitates the addition of new assumptions in order to create a more generic lemma, which we refer to as the {\it {Riemann-Lebesgue lemma in time-average form}}. It is worth mentioning that this lemma is consistent with the original results in undisturbed situations. As for the almost periodic case, we no longer rely on this lemma, but obtain the results by comparing it with the periodic case.

    The present paper is organized as follows: In the second section, we present the symbols, definitions, and one-parameter flows of dynamic systems, as well as a key lemma utilized in this paper: the "Riemann-Lebesgue lemma in time-average form"; In the third section, we expand the observable function by Fourier transformation and deform it in order to determine the limit and integral; In the fourth section, the paper's key findings and supporting evidence are presented; In the fifth part, we generalize the periodic results to the almost periodic case. Also included are two appendices.

    \section{Preliminaries}
    This section is divided into three parts. In the first part, we provide some of the main definitions and notations used in this article. In the second part, we give the dynamic system and its corresponding one-parameter flow and the inverse map of the one-parameter flow  studied in this paper. In the third part, we give the lemmas needed to prove our results.
    It is worth noticing that, in order to condense the article in the following paragraphs, we will sometimes introduce new symbols and provide thorough explanations, which will not influence the comprehension or findings.

    Consider an integrable Hamiltonian system with time $T$-periodic transitions
    \begin{align}\label{HI}
    \begin{bmatrix}
                      \dot{I}\\
                      \dot{\theta}
    \end{bmatrix}\quad &= \quad
    \begin{bmatrix}
        0\\
        \omega\left(I\right)
    \end{bmatrix}, \qquad &t \neq \tau_{k};\\
    \Delta I\quad & = \quad \mathscr{I}_{\tau_{k}}, \qquad &t = \tau_{k};\\
    I\left(0\right)\quad & = \quad   I\left(mT\right), \qquad &m \in \mathbb{N},
    \end{align}
 where $I \in \Omega \subset \mathbb{R}^n,\theta \in \mathbb{T}^n\left(=\mathbb{R}^n/2\pi \mathbb{Z}^n\right)$, $\omega: \Omega_{\infty} \rightarrow \mathbb{T}^{n}$ ,
    \begin{align}\nonumber
        I\left(t\right) &= I\left(\tau_{k-1}\right)+\mathscr{I} _{\tau_{k}},\qquad
        0= \tau_{0} < \tau_{1} < \tau_{2} < \cdots < \tau_{k} \leq t < \tau_{k+1}<\cdots,\\\nonumber
        \mathscr{I}_{\tau_{0}}&=0,
        \end{align}
$\tau_{k} \geq 0 , \ k \in \mathbb{N} $ are constants, and $\Omega_{\infty}$ will be defined later. We assume that the transitions are $T-periodic$, which means: $\exists \ p \in \mathbb{N}^{*}, \exists \ T >0 $ such that $\forall \ i \in \mathbb{N}$
    \begin{align}\nonumber
            \left\{
                \begin{aligned}
          \tau_{i+p} - \tau_{i} &= T,\\
          \mathscr{I}_{\tau_{i+p}} &= \mathscr{I}_{\tau_{i}}.
                \end{aligned}
            \right.
    \end{align}
    It is easy to know that the action variable $I$ in (\ref{HI}) satisfies
    \begin{align}\nonumber
        I\left(t\right) = I\left(0\right) + \sum_{i=0}^{k}\mathscr{I}_{\tau_{i}}, \qquad 0=\tau_{0} < \tau_{k} \leq t~(\mod T)<\tau_{k+1} \leq \tau_{p}.
    \end{align}
    \subsection{Notations and definitions}

    Some basic definitions and notations used in this paper will be given here. One should notice that, unless otherwise specified, the same notations below have the same meanings. Besides, if necessary, we will introduce new notations and make detailed explanations in individual proofs.

    Since the action variable set $\Omega$ will change under the transitions, we note
    \begin{align}\nonumber
        \Omega_{\tau_{k}} = \{I : I_{0}+\sum_{i=0}^{k}\mathscr{I}_{\tau_{i}} , \ I_0 \in \Omega \}, \qquad 0 \leq k \leq p-1.
    \end{align}
    Moreover, because of  the $T$-periodicity of the transitions and the dynamic system we study in this paper, we set
    \begin{align}\nonumber
        \Omega_{\infty} = \bigcup_{i=0}^{p-1} \Omega_{\tau_{i}}.
    \end{align}
    In the following, \ $\forall t > 0 , \ 0 = \tau_{0} < \tau_{k} \leq t~ (\mod{T}) < \tau_{k+1} \leq \tau_{p}$
    \begin{align}\nonumber
        \Omega_{t}=\Omega_{\tau_{k}}.
    \end{align}
    In addition, the meaning of $A_{\tau_{k}} $ and $A_{t}$ for any $A \subset \Omega $ are similar to that of $\Omega_{\tau_{k}} $ and $\Omega_{t}$ respectively.
    In the integrable Hamiltonian system with time $T$-periodic transitions, the changes of the action variable and the angle variable are different from those of the general integrable Hamiltonian system. Take $\left(I,\theta\right)$, and we record the changes of the action variable and the angle variable in a whole period as $\Delta I_{T}$, $\Delta\theta_{T}\left(I\right)$. Then
    \begin{align}\nonumber
        \Delta I_{T} &= \sum_{i=0}^{p-1}\mathscr{I}_{\tau_{k}}, \\
        \Delta \theta_{T}\left(I\right) &= \sum_{i=0}^{p-1}\left(\tau_{i+1}-\tau_{i}\right)\omega\left(I+\sum_{j=0}^{i}\mathscr{I}_{\tau_{j}}\right).\nonumber
    \end{align}
    Moreover we set
    \begin{align}\nonumber
        \Delta \theta_{T}'\left(I\right)=\sum_{i=0}^{p-1}\left(\tau_{i+1}-\tau_{i}\right)\omega_{I}\left(I+\sum_{j=0}^{i}\mathscr{I}_{\tau_{j}}\right),
    \end{align}
    and
    \begin{align}\label{2.10}
        \overline{\omega}_{T}\left(I\right)=\frac{1}{T}{\Delta \theta}_{T}\left(I\right).
    \end{align}
    We call (\ref{2.10}) $the \ average \ change \ of \ the \ angle \ variable \ \theta  \ in \ one \ periodic \ T.$

    One should notice that when $0=\tau_{0} < \tau_{k} \leq t~ (\mod{T}) < \tau_{k+1} \leq \tau_{p}, \ \left(t \gg \tau_{p}\right)$,
    \begin{align}\nonumber
            I\left(t\right) &= I+ \sum_{i=0}^{k} \mathscr{I}_{\tau_{i}},\\
            \theta\left(t\right) &= \theta+m\Delta \theta_{T}\left(I\right)+ \sum_{i=0}^{k-1}\left(\tau_{i+1}-\tau_i\right)\omega\left(I + \sum_{j=0}^{i}\mathscr{I}_{\tau_j}\right)+\left(t-\tau_{k}\right)\omega\left(I+\sum_{i=0}^{k}\mathscr{I}_{\tau_{i}}\right),\nonumber
        \end{align}
        where $\left[\frac{t}{T}\right]= m$, $\left[\cdot\right]$ denotes Gauss function.

    We treat (\ref{HI}) as a random initial value problem. Suppose that the initial condition of (\ref{HI}) is modeled as a pair of dependent random vectors
    described by a joint probability density $f_0 \in L^{1}\left(\Omega \times \mathbb{T}^{n}\right)$. We note that $\varphi _{t} \left(I,\theta\right)$ as the one-parameter flow of system (\ref{HI}) and the
    point $\left(I,\theta\right) = \varphi_{t}\left(I_0,\theta_0\right),$ $\Omega \times \mathbb{T}^{n} \ni  \left(I_0,\theta_0\right)= \left(I\left(0\right),\theta\left(0\right)\right)$, at time $t \in \mathbb{R}^{+}$ is described by the following probability density $f_{t} \in L^{1}\left(\Omega_{t} \times \mathbb{T}^{n}\right):$
    \begin{align}\nonumber
     f_{t}\left(I,\theta\right) = f_{0}\left(\varphi^{-1}_{t}\left(I,\theta\right)\right).
    \end{align}
    Before giving the probability measure induced by the probability density funtion, we first give the definition of the angular average of $F\left(I,\cdot\right) \in L^1\left(\mathbb{T}^{n}\right)$.
    \begin{definition}
        Suppose $F\left(I,\theta\right) \in C \left(M \times \mathbb{T}^{n}\right), M \subset \mathbb{R}^{n},$ and therefore $F\left(I,\cdot\right)$ lies in  $L^{1}\left(\mathbb{T}^{n}\right)$. The angle avreage of $F\left(I,\theta\right)$ is defined as :
        \begin{align}\nonumber
            \overline{F}\left(I,\theta\right) = \frac{1}{\left(2\pi\right)^{n}}\int_{\mathbb{T}^{n}}F\left(I,\theta\right) d \theta.
        \end{align}
    \end{definition}
    \begin{definition} \label{pm}
       Let $ \mathcal{B}\left(\Omega \times \mathbb{T}^n\right) $ denote the $\sigma$- algebra of Borel subsets of $\Omega \times \mathbb{T}^{n}$. Let $\varphi_{t}$ be the one-parameter map given by system (\ref{HI}), $f_{t}$ and $f_{0}$ defined as before. Define corresponding probability measures given by:
        \begin{align}\nonumber
            P_{t}\left(A_{t}\right) = \int_{A_t}f_t\left(I,\theta\right)dId\theta, \  \  \ \overline{P}_{t}\left(A_t\right) = \int_{A_t} \overline{f}_t\left(I,\theta\right)dId\theta, \ \ A \in \mathcal{B}\left(\Omega \times \mathbb{T}^n\right) .
        \end{align}
        Especially, since $\varphi_{t}$ is piecewise constant with respect to the variable $I$  which means : if \ $\left(I,\theta\right) = \varphi_{t}\left(I',\theta '\right)$,
        \begin{align}\nonumber
            \dot{I}\left(t\right)=0, \qquad \ \tau_{i} \leq t~(\mod T) < \tau_{i+1}, \ i =0 ,1 ,2 , \cdots , p-1,
        \end{align}
         we note
        \begin{align}\nonumber
            P_{\tau_{i}}\left(A_{\tau_{i}}\right) = \int_{A_{\tau_{i}}}f_{\tau_{i}}\left(I,\theta\right)dId\theta, \  \  \ \overline{P}_{\tau_{i}}\left(A_{\tau_{i}}\right) = \int_{A_{\tau_{i}}} \overline{f}_{\tau_{i}}\left(I,\theta\right)dId\theta, \ \ A \in \mathcal{B}\left(\Omega \times \mathbb{T}^n\right).
        \end{align}
    \end{definition}
    In the fourth section of this paper, we will give the weak convergence of the probability measure induced by the probability density function under the one-parameter flow $\varphi_{t}$ of system (\ref{HI}). Since the set of action variables $\Omega$ will change under the action of the one-parameter flow $\varphi_{t}$, it is necessary to redefine the weak convergence of the probability measure. First, we review the following definitions :
    \begin{definition}\cite{billingsley2013convergence}
       Let $X$ be a metric space and let $\mathcal{B}\left(X\right)$ denote the $\sigma$- algebra of Borel subsets of $X$. A sequence $\{P_{n}\}$ of probability measures defined on the measurable space $\left(X,\mathcal{B}\left(X\right)\right)$ is said to converge weakly to a probability measure $P$ , also defined on $\left(X,\mathcal{B}\left(X\right)\right)$, if for any $g \in C_{b}\left(X\right)$ we have:
       \begin{align}\nonumber
        \lim_{n \rightarrow +\infty} \int_{X} g dP_{n} = \int_{X} g dP.
       \end{align}
       In this case, we write $P_{n} \Rightarrow P$.
    \end{definition}
    The action spaces of probability measures induced by $f_{0}$ and $f_{t}$ change under the one-parameter flow, but fortunately, under the $T$-periodic transitions (The transitions are only related to time, which makes the change of $I$ with respect to $t$ piecewise constant and periodic.), the change of action variable set $\Omega$ is limited. The change of action variable set $\Omega$ makes it meaningless to directly consider the long-term behavior of probability measures. Therefore, we give the time-average version of the above definition under the action of the one-parameter flow $\varphi_{t}$.
    \begin{definition}
        Let $\mathcal{B}\left(\Omega \times \mathbb{T}^{n} \right)$ denote the $\sigma $- algebra of Borel subsets of $\Omega \times \mathbb{T}^{n}$. Let $P_{t}$ and $\Omega_{t}$ be defined as before. $P_{t}$ which is dependent on $t$ is said to converge weakly to a probability measure $P_{\tau}$ in time-average sense, if for any $g \in C_{b}\left(\Omega_{\infty} \times \mathbb{T}^{n}\right)$, we have:
     \begin{align}\nonumber
        \lim_{t \rightarrow +\infty} \frac{1}{l} \int_{0}^{l} \left[\int_{\Omega_{t} \times \mathbb{T}^{n}}gdP_{t}\right] dt= \int_{\Omega_{\tau}}g dP_{\tau}.
     \end{align}
     In this case, we write $P_{t} \overset{\mathscr{T}}{\Rightarrow} P_{\tau}.$
    \end{definition}
    \begin{remark}
        Under the action of the one-parameter flow $\varphi_{t}$, the probability measure $P_{t}$ is $T$-periodic and piecewise constant which means: $P_{t} = P_{\tau_{k}}, \ 0=\tau_{0} < \tau_{k} \leq t \mod{T} < \tau_{k+1} \leq \tau_{p}$. If
        \begin{align}\nonumber
            \lim_{t \rightarrow +\infty} \frac{1}{T} \int_{0}^{T} \left[\int_{\Omega_{t} \times \mathbb{T}^{n}}gdP_{t}\right] dt= \frac{1}{T} \sum_{i=0}^{p-1}\left(\tau_{i+1}-\tau_{i}\right)\int_{\Omega_{\tau_{k}}}g dP_{\tau_{k}},
        \end{align}
        we write $P_{t} \overset{\mathscr{T}}{\Rightarrow} \frac{1}{T}\sum_{i=0}^{p-1}\left(\tau_{i+1}-\tau_{i}\right)P_{\tau_{i}}$.
    \end{remark}

    In classical statistical mechanics, for a given macrostate $\left(N,V,E\right)$, a statistical system, at any time $t$, is equally likely to be in any one of an extremely large number of distinct microstates. As time passes, the system continually switches from one microstate to another, with the result that, over a reasonable span of time, all one observes is a behavior "averaged" over the variety
    of microstates through which the system passes \cite{pathria2016statistical}.
    \begin{remark}
        The meaning of "$N, V, E$" here is different from that in other parts of this paper. "$N$" represents the number of particles; "$V$" represents the space that the particles confined to; "$E$" represents the total energy of the system.
    \end{remark}
    Given a continuous observable $G$, the behavior of its expected value as an average in space  is an important problem in studying ensembles.
    \begin{definition}\label{def2.5}
         Let $G \in C\left(\Omega_{\infty} \times \mathbb{T}^{n}\right)$. Let $\varphi_{t}$ be the one-parameter flow of system (\ref{HI}), $f_{t}$ and $\Omega_{t}$ defined as before. The expected value of  $G$ under the one-parameter flow $\varphi_{t}$ at time t $ \left(0<\tau_{k} < t~(\mod{T}) \leq \tau_{k+1} \leq \tau_{p} \right)$ is given by:
    \begin{align}\nonumber
        <G>_{t} = \int_{\Omega_{\tau_{k}} \times \mathbb{T}^{n}} G\left(I,\theta\right)f_{t}\left(I,\theta\right)dId\theta.
    \end{align}
    \end{definition}
    The Fourier transform expressed as follows plays an important role in studying  $<G>_{t}$:
    \begin{definition}
         Let $G\in C\left( \Omega_{\infty} \times \mathbb{T}^{n}\right)$. Define function $\hat{G}: \Omega_{\infty} \times \mathbb{Z}^{n} \rightarrow \mathbb{C}$ of $G$:
        \begin{align}\nonumber
            \hat{G}\left(I,\vec{n}\right) = \frac{1}{\left(2\pi\right)^n}\int_{\mathbb{T}^{n}} G\left(I,\theta\right)\exp\left(-\sqrt{-1}<\vec{n},\theta>\right)d\theta.
        \end{align}
    \end{definition}

    \subsection{The one-parameter flow  and its inverse}
    In the above description, it is not difficult to find that the one-parameter flow of system (\ref{HI}) plays an important role in this paper. For the sake of brevity and clarity, we only give the one-parameter flow during $t \gg \tau_{p}$. One can see the specific derivation process of the one-parameter flow in  Appendix A.

    Let $t \gg \tau_{p}, \ \left[\frac{t}{T}\right]= m , \ 0=\tau_{0} < \tau_{k} \leq t~ (\mod{T}) < \tau_{k+1} \leq \tau_{p}$.
    We can define  one-parameter map $\varphi_{t}$ on $\Omega\times \mathbb{T}^n$ :
    \begin{align}\nonumber
            \varphi _t : \quad \Omega \times \mathbb{T}^n & \rightarrow \Omega_{\tau_{k}} \times \mathbb{T}^n, \\
            \left(I,\theta\right)  &\rightarrow \left(I+\sum_{i=0}^{k}\mathscr{I}_{\tau_{i}},\theta+m\Delta \theta_{T}\left(I\right)+ \sum_{i=0}^{k-1}\left(\tau_{i+1}-\tau_i\right)\omega\left(I + \sum_{j=0}^{i}\mathscr{I}_{\tau_j}\right)+\left(t-\tau_{k}\right)\omega\left(I+\sum_{i=0}^{k}\mathscr{I}_{\tau_{i}}\right)\right).\nonumber
    \end{align}
    Note that the Jacobian matrix of $\varphi_{t}$:
    \begin{align}\nonumber
        D\varphi_{t}=
        \begin{bmatrix}
            \bm{1_n}&0\\
            m\Delta \theta_{T}'\left(I\right)+ \sum_{i=0}^{k-1}\left(\tau_{i+1}-\tau_i\right)\omega_{I}\left(I + \sum_{j=0}^{i}\mathscr{I}_{\tau_j}\right)+\left(t-\tau_{k}\right)\omega_{I}\left(I+\sum_{i=0}^{k}\mathscr{I}_{\tau_{i}}\right)&\bm{1_n}.
        \end{bmatrix},
        \qquad det\left(D\varphi_{t}\right) = 1.
    \end{align}

    Next, we solve the inverse map of the one-parameter flow,  which plays an important role in calculating determinant of coordinate transformation, when variable replacement formula is used in integral operation.

    The inverse map of $\varphi_{t}$ is defined as :
    \begin{align}\nonumber
            \varphi _t^{-1} : \quad \Omega_{\tau_{k}} \times \mathbb{T}^n & \rightarrow \Omega \times \mathbb{T}^n \\
            \left(I,\theta\right)  &\rightarrow\left(\varphi _t^{-1}\left(I,\theta\right)_{\left(I\right)},\varphi _t^{-1}\left(I,\theta\right)_{\left( \theta \right)}\right).\nonumber
    \end{align}
    On the one hand,
    \begin{align}\nonumber
        \varphi_{t}\circ \varphi_{t}^{-1}\left(I,\theta\right) = \left(I,\theta\right),
    \end{align}
    on the other hand,
    \begin{align}\nonumber
        &\varphi_{t}\left(\varphi _t^{-1}\left(I,\theta\right)_{\left(I\right)},\varphi _t^{-1}\left(I,\theta\right)_{\left( \theta \right)}\right)\\\nonumber
        &=( \varphi _t^{-1}\left(I,\theta\right)_{\left(I\right)}+\sum_{i=0}^{k}\mathscr{I}_{\tau_{i}},\\\nonumber
        &\varphi _t^{-1}\left(I,\theta\right)_{\left( \theta \right)}+m\Delta \theta_{T}\left(\varphi _t^{-1}\left(I,\theta\right)_{\left(I\right)}\right)+ \sum_{i=0}^{k-1}\left(\tau_{i+1}-\tau_i\right)\omega\left(\varphi _t^{-1}\left(I,\theta\right)_{\left(I\right)}+ \sum_{j=0}^{i}\mathscr{I}_{\tau_j}\right)+\left(t-\tau_{k}\right)\omega\left(\varphi _t^{-1}\left(I,\theta\right)_{\left(I\right)}+\sum_{i=0}^{k}\mathscr{I}_{\tau_{i}}\right)),\nonumber
        \end{align}
    which means
    \begin{align}\nonumber
        \left\{
        \begin{aligned}
            \varphi _t^{-1}\left(I,\theta\right)_{\left(I\right)} &=I-\sum_{i=0}^{k}\mathscr{I}_{\tau_{i}} ,\\
            \varphi _t^{-1}\left(I,\theta\right)_{\left( \theta \right)} &=\theta - m\Delta \theta_{T}\left(I-\sum_{i=0}^{k}\mathscr{I}_{\tau_{i}}\right)- \sum_{i=0}^{k-1}\left(\tau_{i+1}-\tau_i\right)\omega\left(I-\sum_{j=i+1}^{k}\mathscr{I}_{\tau_{j}}\right)-\left(t-\tau_{k}\right)\omega\left(I\right).
        \end{aligned}
        \right.
    \end{align}
    Now we get the inverse map $\varphi_{t}^{-1}$ on $\Omega_{\tau_{k}} \times \mathbb{T}^{n}$:
    \begin{align}\nonumber
        \begin{aligned}
            \varphi _t^{-1} : \quad \Omega_{\tau_{k}} \times \mathbb{T}^n & \rightarrow \Omega \times \mathbb{T}^n \\
            \left(I,\theta\right)  &\rightarrow\left(I-\sum_{i=0}^{k}\mathscr{I}_{\tau_{i}} ,\theta - m\Delta \theta_{T}\left(I-\sum_{i=0}^{k}\mathscr{I}_{\tau_{i}}\right)- \sum_{i=0}^{k-1}\left(\tau_{i+1}-\tau_i\right)\omega\left(I-\sum_{j=i+1}^{k}\mathscr{I}_{\tau_{j}}\right)-\left(t-\tau_{k}\right)\omega\left(I\right)\right).
        \end{aligned}
    \end{align}
    Moreover
    \begin{align}\nonumber
        D\varphi_t^{-1}=
        \begin{bmatrix}
            \bm{1_n}&0\\
            - m\Delta \theta_{T}'\left(I-\sum_{i=0}^{k}\mathscr{I}_{\tau_{i}}\right)- \sum_{i=0}^{k-1}\left(\tau_{i+1}-\tau_i\right)\omega_{I}\left(I-\sum_{j=i+1}^{k}\mathscr{I}_{\tau_{j}}\right)-\left(t-\tau_{k}\right)\omega_{I}\left(I\right)&\bm{1_n}
        \end{bmatrix},
    \end{align}
    \begin{equation*}
        det\left(D\varphi_t^{-1}\right)=1.
    \end{equation*}

    \subsection{Lemmas}

    In this part , we give the key lemma to prove the main results in this paper : $Riemann-Lebesgue \ lemma  \ in \ time-average \ form$ to prove the results of this paper, and we also give some other necessary lemmas. These lemmas are mentioned in \cite{mitchell2019weak}, but in this paper, the integrable Hamiltonian system is disturbed by transitions, so when considering its ensemble problem, we must make some modifications.
    Some well-known theorems used in this paper are given in Appendix B.

    We first give a simple lemma, whose proof is basic.
    \begin{lemma}\label{321}
       Assume $M\left(t\right)$ is a bounded function defined on $\mathbb{R}^{+}$ with bound $M \gg 1$ and
        $$\lim_{t \rightarrow +\infty} M\left(t\right) = 0. $$
        Then
        \begin{equation*}
            \lim_{l \rightarrow +\infty} \frac{1}{l}\int_{0}^{l}M\left(t\right)dt = 0.
        \end{equation*}
    \end{lemma}
    \begin{proof}
        By the assumptions of the $M\left(t\right)$ : $\forall \epsilon > 0$
        $$\exists \  t'_0\left(\epsilon\right) >0,  \quad s.t.  \ \forall t > t'_0\left(\epsilon \right), \ |M\left(t\right)| < \frac{\epsilon}{M}.$$
    Take $t_0 \left(\epsilon\right)= max \left\{t'_0\left(\epsilon\right), \left(M^2-\epsilon\right)t'_0\left(\epsilon\right)/ \left[\left(M-1\right)\epsilon\right]\right\}$, and we get: \ $\forall l > t_0\left(\epsilon\right)$,
    \begin{align}\nonumber
        \begin{aligned}
        \frac{1}{l}\big| \int_{0}^{l} M\left(t\right) dt \ \big| &\leq \frac{1}{l}\big| \left(\int_{0}^{t_0\left( \epsilon \right)} + \int_{t_0\left(\epsilon\right)}^{l}\right) M\left(t\right) dt \ \big| \\
    &\leq \frac{1}{l}\int_{0}^{t_0\left(\epsilon\right)}\big | M\left(t\right)\big | dt + \frac{1}{l}\int_{t_0\left(\epsilon\right)}^{l}\big | M\left(t\right)\big | dt\\
    &\leq Mt_0\left(\epsilon\right)/l + \left(\epsilon / M \right)\left[l-t_0\left(\epsilon\right)\right]/l\\
    &< \epsilon.
        \end{aligned}
    \end{align}
    \end{proof}
    We call the following lemma $ Riemann-Lebesgue \ lemma \ in  \ time-average \ form $  which is the linchpin to prove our main results in this paper. One can see the classical $Riemann-Lebesgue~ lemma$ in Appendix B.
    \begin{lemma}\label{il}
        Let $\Omega \subset \mathbb{R}^n$ be open. Suppose that $F\left(I,t\right)$ and $\phi \left(I,t\right)$ satisfy the following:

        $\bm{\left(1\right)}$ For any fixed $t>0$, $F\left(I,t\right) \in \bm{L}^1\left(\Omega\right)$;

        $\bm{\left(2\right)}$ $F\left(I,t\right)$ is a periodic function on variable $t$, which means there exists $T>0$ such that $F\left(I,t+T\right)= F\left(I,t\right)$
        $\forall \left(I,t\right) \in \Omega \times \mathbb{R}^{+}$;

        $\bm{\left(3\right)}$ $F\left(I,t\right)$ is piecewise constant with respect to variable t for any fixed I,

        $ i.e. \ \exists \ 0 = \tau_0 <\tau_1 <\tau_2<\cdots<\tau_p, \tau_p - \tau_0 = T \ s.t.$
        \begin{align}\nonumber
            \begin{aligned}
                F\left(I,t\right)=F_1\left(I\right),  \qquad&  \tau_0 \leq t < \tau_1\\
                F\left(I,t\right)=F_2\left(I\right),  \qquad&  \tau_1 \leq t < \tau_2\\
                \vdots \qquad&\\
                F\left(I,t\right)=F_p\left(I\right),  \qquad&  \tau_{p-1} \leq t < \tau_p
            \end{aligned}
        \end{align}
        $\forall I \in \Omega$;

        $\bm{\left(4\right)}$ $\phi\left(I,t\right) \in \bm{C}^2\left(\Omega\right)$ and we can find an $\omega \left(I\right) \in \bm{C}^2\left(\Omega\right), \nabla_{I} \omega\left(I\right) \neq 0$ such that
        \begin{equation*}
            \frac{1}{t}\int_{0}^{t}\phi\left(I,\tau\right)d\tau-\omega\left(I\right) = \bm{o}\left(1/t\right).
        \end{equation*}
         Then
        \begin{equation*}
            \lim_{l \rightarrow +\infty} \frac{1}{l}\int_{0}^{l}\int_{\Omega}F\left(I,t\right)\exp\left(\sqrt{-1}\int_{0}^{t}\phi\left(I,\tau\right)d\tau\right)dIdt=0.
        \end{equation*}
    \end{lemma}
    \begin{proof}

        We set
        \begin{equation*}
            M\left(t\right) = \int_{\Omega}F\left(I,t\right)\exp\left(\sqrt{-1}\int_{0}^{t}\phi\left(I,\tau\right)d\tau\right)dI.
        \end{equation*}
        The lemma\ref{321} implies that we only need to prove $M\left(t\right) = 0 , as \ t \rightarrow +\infty$.

    Conbining assumptions $\bm{\left(1\right)}$ and $\bm{\left(3\right)}$ gives : $$\forall \ t>0, \ \exists \ 1\leq j \leq p, \quad s.t.\ F\left(I,t\right)= F_j\left(I\right), \ F_{j} \in \bm{L}^1\left(\Omega\right).$$
        For any $F_i\left(I\right) \in \bm{L}^1\left(\Omega\right)$, we prove the special case $\tilde{F}_{i}\left(I\right) \in \bm{C}_{c}^{\infty}\left(\Omega\right)$ first.

        It is easy to see, $\tilde{F}_{i}\left(I\right) \in \bm{C}_{c}^{\infty}\left(\Omega\right) \subset \bm{L}^1\left(\Omega\right)$, and there is a bounded and closed set $K_{i} \subset \Omega$, which depends on $i$, such
        that $\tilde{F}_i\left(I\right) = 0$ on $K_{i}^{c}$, where $K_{i}^{c} = \Omega \verb|\| K_{i}.$ Then lemma 0 in article \cite{mitchell2019weak} yields:
        \begin{equation*}
            \forall \epsilon > 0, \ \exists \  t_{\omega}^{i}\left(\epsilon\right) > 0,\quad s.t. \ \forall  t>t_{\omega}^{i}\left(\epsilon\right) ,  \quad  \big| \int_{\Omega}\tilde{F}_i\left(I\right)\exp\left(\sqrt{-1}\omega\left(I\right)t\right)dI\big| < \epsilon/4.
        \end{equation*}
        Moreover. using assumptions $\bm{\left(4\right)}$ and the continuity of the exponential function gives
        \begin{equation*}
            \exists \delta_i\left(\epsilon\right) > 0 , \ \exists t^i\left(\epsilon\right)>0, \quad s.t. \ \forall t >t^i\left(\epsilon\right), \ \big| \sqrt{-1}\frac{1}{t}\int_{0}^{t}\phi\left(I,\tau\right)d\tau \cdot t - \omega\left(I\right) t \big| < \delta_{i}\left(\epsilon\right),
        \end{equation*}
        \begin{equation*}
            \big| \ \exp\left[\sqrt{-1}\frac{1}{t}\int_{0}^{t}\phi\left(I,\tau\right)d\tau \cdot t\right] - \exp\left[\sqrt{-1}\omega\left(I\right)t\right] \ \big| < \epsilon / \left[4m\left(K_{i}\right)F_i\right],
        \end{equation*}
        where $\sup \limits_{I \in \Omega}\tilde{F}_{i}\left(I\right)= \sup \limits_{I \in K_{i}}\tilde{F}_{i}\left(I\right) \coloneqq F_{i}.$

        Set $t_i\left(\epsilon\right) = \max \left\{t_{\omega}^i\left(\epsilon\right),t^i\left(\epsilon\right)\right\}$. Then $\forall \ \tilde{F}_{i}\left(I\right) \in \bm{C}_c^{\infty}\left(\Omega\right), \ \forall \epsilon >0, \ \exists\ t_i\left(\epsilon\right) >0 \quad  s.t. \ \forall \ t>t_i\left(\epsilon\right)$
        \begin{equation}\label{144}
            \big| \int_{\Omega}\tilde{F}_{i}\left(I\right)\exp\left(\sqrt{-1}\int_{0}^{t}\phi\left(I,\tau\right)d\tau\right) dI \big| < \epsilon/2.
        \end{equation}
    Actually,
    \begin{align}\nonumber
        \begin{aligned}
      &\big| \int_{\Omega}\tilde{F}_{i}\left(I\right)\exp\left(\sqrt{-1}\int_{0}^{t}\phi\left(I,\tau\right)d\tau\right) dI \big| \\
     &\leq \big| \int_{K_{i}}\tilde{F}_{i}\left(I\right)\exp\left(\sqrt{-1}\int_{0}^{t}\phi\left(I,\tau\right)d\tau\right) dI \big| \\
     &\leq \big| \int_{K_{i}}\tilde{F}_{i}\left(I\right)\left[\exp\left(\sqrt{-1}\int_{0}^{t}\phi\left(I,\tau\right)d\tau\right)-\exp\left(\sqrt{-1}K_{i}\left(I\right)t\right)\right] dI \big|+\big| \int_{K_{i}}\tilde{F}_{i}\left(I\right)\exp\left(\sqrt{-1}\omega\left(I\right)t\right)dI\big|\\
     &\leq \big| \exp\left(\sqrt{-1}\int_{0}^{t}\phi\left(I,\tau\right)d\tau\right)-\exp\left(\sqrt{-1}\omega\left(I\right)t\right) \big|\big|\int_{K_{i}}\tilde{F}_{i}\left(I\right)dI \big|+\epsilon/4\\
     &\leq \big| \exp\left(\sqrt{-1}\int_{0}^{t}\phi\left(I,\tau\right)d\tau\right)-\exp\left(\sqrt{-1}\omega\left(I\right)t\right) \big|\big|\int_{K_{i}}dI \big| \ \cdot  \ \big| F_{i} \big| +\epsilon/4\\
     &< \epsilon / \left[4m\left(K_{i}\right)F_{i}\right]\cdot m\left(K_{i}\right) \cdot F_{i}+ \epsilon/4 = \epsilon/2.
        \end{aligned}
    \end{align}
    Secondly, using the density of $\bm{C}_{c}^{\infty}\left(\Omega\right)$ in $\bm{L}^1\left(\Omega\right)$, we take a sequence $\tilde{F}^{k}_{i}\left(I\right) \in \bm{C}_{c}^{\infty}\left(\Omega\right)$ such that
    \begin{equation}\label{146}
        \forall \ \epsilon >0, \ \exists\ n_{0}\left(\epsilon\right), \forall k\geq n_{0}\left(\epsilon\right), \ \int_{\Omega} \big| F_{i}\left(I\right) - \tilde{F}^{k}_{i}\left(I\right) \big| dI < \epsilon / 2,
    \end{equation}
    and for $\tilde{F}^{k}_{i}\left(I\right)$,  we can also take a $t_{i}^{n_{0}}\left(\epsilon\right)$ to satisfy (\ref{144}), using the previous discussion.

    Then
    $\forall F_{i} \in \bm{L}^{1}\left(\Omega\right), \ \forall \ \epsilon >0,  \exists \ t_{i}^{n_{0}}\left(\epsilon\right)>0,\  s.t. \ \forall t>t_{i}^{n_{0}}\left(\epsilon\right)$
    \begin{align}\nonumber
        \begin{aligned}
       &\big| \int_{\Omega}F_i\left(I\right)\exp\left(\sqrt{-1}\int_{0}^{t}\phi\left(I,\tau\right)d\tau\right)dI\big| \\
       &\leq \int_{\Omega} \big| F_{i}\left(I\right) - \tilde{F}^{n_{0}}_{i}\left(I\right) \big| \big| \exp\left(\sqrt{-1}\int_{0}^{t}\phi\left(I,\tau\right)d\tau\right) \big|dI +\big| \int_{\Omega}\tilde{F}_{i}^{n_{0}}\left(I\right)\exp\left(\sqrt{-1}\int_{0}^{t}\phi\left(I,\tau\right)d\tau\right) dI \big|\\
       &< \epsilon/2 + \epsilon/2 = \epsilon.
    \end{aligned}
    \end{align}

    To be concluded, $\forall \epsilon>0, \ \exists \ t\left(\epsilon \right) = \max \limits_{1\leq i \leq p} \left\{t_i^{n_0}\left(\epsilon\right)\right\}>0, \ \forall t > t\left(\epsilon\right),$
    \begin{equation*}
        \big| \int_{\Omega}F\left(I,t\right)\exp\left(\sqrt{-1}\int_{0}^{t}\phi\left(I,\tau\right)d\tau\right)dI \big| = \big|\int_{\Omega}F_j\left(I\right)\exp\left(\sqrt{-1}\int_{0}^{t}\phi\left(I,\tau\right)d\tau\right)dI \big|<\epsilon.
    \end{equation*}
    \end{proof}

   We also need the following two lemmas  to ensure that limiting operations can be exchanged freely, however, we have to 
    deal with the transitions in system (\ref{HI}). 
    \begin{lemma}\label{il13}
        Let $\Omega \subset \mathbb{R}$ be open and let $\mathscr{I}\left(t\right)$ be a bounded function and $G\left(I,\theta\right) \in \bm{C}_{b}\left(\Omega_{\mathscr{I}} \times \mathbb{T}^{n}\right)$ with
        \begin{equation*}
            \hat{G}\left(I+\mathscr{I}\left(t\right),\vec{n}\right) = \frac{1}{\left(2\pi\right)^{n}}\int_{\mathbb{T}^{n}}G\left(I+\mathscr{I}\left(t\right),\theta\right)\exp\left(-\sqrt{-1}<\vec{n},\theta>\right)d\theta,
        \end{equation*}
        where $\Omega_\mathscr{I} = \{ x : x = I + \mathscr{I}\left(t\right), I \in \Omega , t \in \mathbb{R}^{+}\}$.
Then for each $\vec{n} \in \mathbb{Z}^{n}$, the function $\hat{G}\left(\cdot + \mathscr{I}\left(t\right),\vec{n}\right)$ lies in $\bm{C}_{b}\left(\Omega\right)$.
    \end{lemma}
    \begin{proof}
        We can get the boundedness of the function $\hat{G}\left(I+\mathscr{I}\left(t\right),\vec{n}\right)$ easily by
        \begin{align}\nonumber
            \begin{aligned}
            |\hat{G}\left(I+\mathscr{I}\left(t\right),\vec{n}\right)| & \leq \frac{1}{\left(2\pi\right)^{n}}\int_{\mathbb{T}^{n}}|G\left(I+\mathscr{I}\left(t\right),\theta\right)| |\exp\left(-\sqrt{-1}<\vec{n},\theta>\right)| d\theta,\\
            & \leq  \sup_{\left(I,\theta\right) \in \bar{\Omega}_{\mathscr{I}} \times \mathbb{T}^{n}} |G\left(I,\theta\right)| \frac{1}{\left(2\pi\right)^{n}} \int_{\mathbb{T}^{n}}  1 \ d\theta < \infty
            \end{aligned}
        \end{align}
        In addtion, $\forall I_{0} \in \Omega$, suppose $\left\{ I_{i}\right\}_{i=1}^{\infty} \subset \Omega $ such that $I_{i} \rightarrow I_{0}$ as $i \rightarrow \infty$, then
        \begin{equation*}
            G\left(I_{i},\theta\right) \rightarrow G\left(I_0,\theta\right), \qquad i \rightarrow \infty,
        \end{equation*}
        so, \ $\forall t \in \mathbb{R}^{+}$
        \begin{equation*}
            G\left(I_{i}+\mathscr{I}\left(t\right),\theta\right) \rightarrow G\left(I_0+\mathscr{I}\left(t\right),\theta\right), \qquad i \rightarrow \infty.
        \end{equation*}
        Since $|G\left(I_{i}+\mathscr{I}\left(t\right),\theta\right)| < \infty$, using the Lebesgue's dominated convergence theorem we have
        \begin{align}\nonumber
            \begin{aligned}
            \lim_{i \rightarrow +\infty}\hat{G}\left(I_i+\mathscr{I}\left(t\right),\vec{n}\right) & =  \lim_{i \rightarrow +\infty}\frac{1}{\left(2\pi\right)^{n}}\int_{\mathbb{T}^{n}}G\left(I_{i}+\mathscr{I}\left(t\right),\theta\right)\exp\left(-\sqrt{-1}<\vec{n},\theta>\right)d\theta,\\
            & = \frac{1}{\left(2\pi\right)^{n}}\int_{\mathbb{T}^{n}} \lim_{i \rightarrow +\infty} G\left(I_{i}+\mathscr{I}\left(t\right),\theta\right)\exp\left(-\sqrt{-1}<\vec{n},\theta>\right)d\theta,\\
            & = \frac{1}{\left(2\pi\right)^{n}}\int_{\mathbb{T}^{n}} G\left(I+\mathscr{I}\left(t\right),\theta\right)\exp\left(-\sqrt{-1}<\vec{n},\theta>\right)d\theta,\\
            & = \hat{G}\left(I+\mathscr{I}\left(t\right),\vec{n}\right).
            \end{aligned}
        \end{align}
    This proves the lemma.
    \end{proof}
    \begin{lemma}\label{il14}
        Suppose $\mathscr{I}\left(t\right)$ is a bounded function, $G\left(I,\theta\right) \in \bm{C}_b \left( \Omega_\mathscr{I} \times \mathbb{T}^n\right)$ and $f_0\left(I,\theta\right) \in \bm{C}_c \left( \Omega_\mathscr{I} \times \mathbb{T}^n \right).$ Then
        \begin{equation*}
            \int_{\Omega}dI \sum_{\vec{n} \in \mathbb{Z}^n} |\hat{G}\left(I + \mathscr{I} \left(t\right),\vec{n}\right)\hat{f}_0\left(I,-\vec{n}\right)| < \infty.
        \end{equation*}

    \end{lemma}
    The  proof of lemma \ref{il14} is quite similar to that in article \cite{mitchell2019weak}, one can see more details in it.


    \section{The expected value of observable function}
    In this paper, we use the long-time behavior of expected value at time $t$, $<G>_{t}$ ,  of observable function $G$ to describe the ensemble problem of system (\ref{HI}) and the convergence problem of probability measure. In the process of analysis, Fourier transform plays an important role. We will give the Fourier transform form of $<G>_{t}$ in detail here.

    Suppose $\tau_{0} \ll \tau_{k} \leq t~ (\mod{T}) < \tau_{k+1} \leq \tau_{p}, \ m=[\frac{t}{T}]$. Combining definition \ref{def2.5} and one-parameter flow $\varphi_{t}$, we can get :
    \begin{align}\nonumber
        \begin{aligned}
        <G>_t &= \int_{\Omega_{\tau_{k}} \times \mathbb{T}^n } G\left(I,\theta\right)f_t\left(I,\theta\right)dId\theta \\
        &= \int_{\Omega_{\tau_{k}} \times \mathbb{T}^n } G\left(I,\theta\right)f_0\left(\varphi_{t}^{-1}\left(I,\theta\right)\right)dId\theta.
        \end{aligned}
    \end{align}
    Using the $variable\ substitution\ formula$ : $\left(I_{0},\theta_{0}\right) = \varphi_{t}^{-1}\left(I,\theta\right)$ yields
    \begin{align}\nonumber
        \begin{aligned}
            <G>_t
        &= \int_{\Omega \times \mathbb{T}^n } G\left(\varphi_{t}\left(I_0,\theta_0\right)\right)f_0\left(I_0,\theta_{0}\right)det\left(\varphi_{t}\left(I_{0},\theta_{0}\right)\right) dI_{0}d\theta_{0}\\
        &= \int_{\Omega \times \mathbb{T}^n } G\left(\varphi_{t}\left(I_0,\theta_0\right)\right)f_0\left(I_0,\theta_{0}\right)dI_{0}d\theta_{0}\\
        &= \int_{\Omega \times \mathbb{T}^n } G\left(I+\sum_{i=0}^{k}\mathscr{I}_{\tau_{i}},\theta+m\Delta \theta_{T}\left(I\right)+ \sum_{i=0}^{k-1}\left(\tau_{i+1}-\tau_i\right)\omega\left(I + \sum_{j=0}^{i}\mathscr{I}_{\tau_j}\right)+\left(t-\tau_{k}\right)\omega\left(I+\sum_{i=0}^{k}\mathscr{I}_{\tau_{i}}\right)\right)f_0
        \left(I_0,\theta_{0}\right)dI_{0}d\theta_{0}.
        \end{aligned}
    \end{align}
    Furthermore, the Fourier coefficients of the function $G\left(\varphi_{t}\left(I,\theta\right)\right)$ are given by:
    \begin{align}\nonumber
        \begin{aligned}
        &\frac{1}{\left(2 \pi\right)^n}\int_{\mathbb{T}^n}G\left(I+\sum_{i=0}^{k}\mathscr{I}_{\tau_{i}},\theta+m\Delta \theta_{T}\left(I\right)+ \sum_{i=0}^{k-1}\left(\tau_{i+1}-\tau_i\right)\omega\left(I + \sum_{j=0}^{i}\mathscr{I}_{\tau_j}\right)+\left(t-\tau_{k}\right)\omega\left(I+\sum_{i=0}^{k}\mathscr{I}_{\tau_{i}}\right)\right)\exp{-\sqrt{-1}<\vec{n},\theta>}d\theta\\
        &=\hat{G}\left(I+\sum_{i=0}^{k}\mathscr{I}_{\tau_{i}},\vec{n}\right)\exp\left[{\sqrt{-1}<\vec{n},m\Delta \theta_{T}\left(I\right)+ \sum_{i=0}^{k-1}\left(\tau_{i+1}-\tau_i\right)\omega\left(I + \sum_{j=0}^{i}\mathscr{I}_{\tau_j}\right)+\left(t-\tau_{k}\right)\omega\left(I+\sum_{i=0}^{k}\mathscr{I}_{\tau_{i}}\right)>}\right].
        \end{aligned}
    \end{align}
    Parseval's theorem implies that for each\ $I \in \Omega$:
    \begin{align}\nonumber
        \begin{aligned}
        &\frac{1}{\left(2\pi\right)^n}\int_{\mathbb{T}^n}G\left(I+\sum_{i=0}^{k}\mathscr{I}_{\tau_{i}},\theta+m\Delta \theta_{T}\left(I\right)+ \sum_{i=0}^{k-1}\left(\tau_{i+1}-\tau_i\right)\omega\left(I + \sum_{j=0}^{i}\mathscr{I}_{\tau_j}\right)+\left(t-\tau_{k}\right)\omega\left(I+\sum_{i=0}^{k}\mathscr{I}_{\tau_{i}}\right)\right)
        \left[f_0\left(I,\theta\right)\right]^{*}d\theta\\
        &=\sum_{\vec{n} \in \mathbb{Z}^n}\hat{G}\left(I+\sum_{i=0}^{k}\mathscr{I}_{\tau_{i}},\vec{n}\right)\left[\hat{f}_0\left(I,\vec{n}\right)\right]^{*}\exp\left[{\sqrt{-1}<\vec{n},m\Delta \theta_{T}\left(I\right)+ \sum_{i=0}^{k-1}\left(\tau_{i+1}-\tau_i\right)\omega\left(I + \sum_{j=0}^{i}\mathscr{I}_{\tau_j}\right)+\left(t-\tau_{k}\right)\omega\left(I+\sum_{i=0}^{k}\mathscr{I}_{\tau_{i}}\right)>}\right].
        \end{aligned}
    \end{align}
    Here the $*$ denotes complex conjungation. Noting that
    \begin{align}\nonumber
        \begin{aligned}
       \left[\hat{f}_0\left(I,\vec{n}\right)\right]^{*}&=\left[\frac{1}{\left(2 \pi\right)^n} \int_{T^n}f_0\left(I,\theta\right)\exp{-\sqrt{-1}<\vec{n},\theta>}d\theta\right]^{*}\\
    &=\frac{1}{\left(2 \pi\right)^n} \int_{\mathbb{T}^n}\left[f_0\left(I,\theta\right)\right]^{*}\left[\exp{-\sqrt{-1}<\vec{n},\theta>}\right]^{*}d\theta\\
    &=\frac{1}{\left(2 \pi\right)^n} \int_{\mathbb{T}^n}f_0\left(I,\theta\right)\left[\cos\left(-<\vec{n},\theta>\right)+\sqrt{-1}\sin\left(<-\vec{n},\theta>\right)\right]^{*}d\theta\\
    &=\frac{1}{\left(2 \pi\right)^n} \int_{\mathbb{T}^n}f_0\left(I,\theta\right)\left[\cos\left(-<-\vec{n},\theta>\right)+\sqrt{-1}\sin\left(-<-\vec{n},\theta>\right)\right]d\theta\\
    &=\frac{1}{\left(2 \pi\right)^n} \int_{\mathbb{T}^n}f_0\left(I,\theta\right)\exp{-\sqrt{-1}<-\vec{n},\theta>}d\theta\\
    &=\hat{f}_0\left(I,-\vec{n}\right).
        \end{aligned}
    \end{align}
    Hence
    \begin{align}\nonumber
        \begin{aligned}
            \begin{aligned}
        <G>_t=\left(2\pi\right)^{n}\int_{\Omega}\sum_{\vec{n} \in \mathbb{Z}^n}\{&\hat{G}\left(I+\sum_{i=0}^{k}\mathscr{I}_{\tau_{i}},\vec{n}\right)\hat{f}_0\left(I,-\vec{n}\right)\\\nonumber
        &\exp\left[{\sqrt{-1}<\vec{n},m\Delta \theta_{T}\left(I\right)+ \sum_{i=0}^{k-1}\left(\tau_{i+1}-\tau_i\right)\omega\left(I + \sum_{j=0}^{i}\mathscr{I}_{\tau_j}\right)+\left(t-\tau_{k}\right)\omega\left(I+\sum_{i=0}^{k}\mathscr{I}_{\tau_{i}}\right)>}\right]\}dI,\nonumber
            \end{aligned}
    \end{aligned}
    \end{align}
    \begin{align}\label{expect}
        \frac{1}{\left(2\pi\right)^{n}}<G>_t=\int_{\Omega}&\hat{G}\left(I+\sum_{i=0}^{k}\mathscr{I}_{\tau_{i}},\vec{0}\right)
        \hat{f}_0\left(I,\vec{0}\right)dI+\int_{\Omega}\sum_{\substack{\vec{n} \neq \vec{0}\\\nonumber
        \vec{n} \in \mathbb{Z}^n}}\hat{G}\left(I+\sum_{i=0}^{k}\mathscr{I}_{\tau_{i}},\vec{n}\right)\hat{f}_0\left(I,-\vec{n}\right)\\
        &\exp\left[{\sqrt{-1}<\vec{n},m\Delta \theta_{T}\left(I\right)+ \sum_{i=0}^{k-1}\left(\tau_{i+1}-\tau_i\right)\omega\left(I + \sum_{j=0}^{i}\mathscr{I}_{\tau_j}\right)+\left(t-\tau_{k}\right)\omega\left(I+\sum_{i=0}^{k}\mathscr{I}_{\tau_{i}}\right)>}\right]dI.
    \end{align}

    We consider the time-average of $<G>_t$ for sufficiently large $l$:
    \begin{equation}\label{222}
        \frac{1}{lT}\int_0^{lT}<G>_t dt.
    \end{equation}


    \section{Main results}
    In this section, we give the main results and their proofs.

    \begin{theorem}\label{importanttheorem1}
        In system (\ref{HI}), suppose that $\omega\left(I\right) \in \bm{C}^2\left(\Omega_{\infty}\right)$ and $\overline{\omega}_{T}\left(I\right) \in \bm{C}^2\left(\Omega\right)$ has no critical points, and suppose the initial condition is described by a probability density function $f_0 \in \bm{L}^{1}\left(\Omega \times \mathbb{T}^{n}\right)$. For any continuous and bounded function $G\in \bm{C}_b\left(\Omega_{\infty}\times \mathbb{T}^n\right),$   (\ref{222}) has limit:
        \begin{equation*}
            \displaystyle\lim_{l \rightarrow +\infty}\frac{1}{lT}\int_0^{{lT^{-}}}<G>_t dt = \frac{1}{T}\sum_{i=0}^{p-1}\left(\tau_{i+1}-\tau_{i}\right)<\bar{G}\left(I+\sum_{j=0}^{i}\mathscr{I}_{\tau_{j}}\right)>_0.
        \end{equation*}
    \end{theorem}
    \begin{proof}

        We prove the theorem in a special condition $f_{0} \in \bm{C}_{c}\left(\Omega \times \mathbb{T}^{n}\right)$ as begin.

        Because the expected value of the observable function under the one-parameter flow $<G>_{t}$ is involved in the integral operation, we introduce $k\left(t\right)$ and $m\left(t\right)$ to represent the $k$ and $m$ that changes with $t$ respectively. The $k$ satisfies $\tau_{0} \leq \tau_{k} \leq t~ (\mod{T}) < \tau_{k+1} \leq \tau_{p}$ and $m=[\frac{t}{T}]$.

        The first part of (\ref{222}) \ $\left(\vec{n} = \vec{0}\right)$ \ is:
    \begin{align}\nonumber
        \begin{aligned}
            &\left(2\pi\right)^{n}\frac{1}{lT}\int_{0}^{lT^{-}}\int_{\Omega} \hat{G}\left(I+\sum_{i=0}^{k\left(t\right)}\mathscr{I}_{\tau_{i}},\vec{0}\right)\hat{f}_0\left(I,\vec{0}\right)dIdt \\
           & =\left(2\pi\right)^{n}\frac{1}{lT}\int_{0}^{lT^{-}}\int_{\Omega} \left[\frac{1}{\left(2\pi\right)^{n}} \int_{\mathbb{T}^n}G\left(I+\sum_{i=0}^{k\left(t\right)}\mathscr{I}_{\tau_{i}},\theta\right)d\theta\right]\left[\frac{1}{\left(2\pi\right)^{n}} \int_{\mathbb{T}^n}f_0\left(I,\theta\right)d\theta\right]dIdt \\
           & =\frac{1}{lT}\int_{0}^{lT^{-}}\int_{\Omega}\int_{\mathbb{T}^n}\bar{G}\left(I+\sum_{i=0}^{k\left(t\right)}\mathscr{I}_{\tau_{i}}\right)f_0\left(I,\theta\right)d\theta dI dt,\\
        \end{aligned}
    \end{align}
    where $\bar{G}\left(I\right)=\frac{1}{\left(2\pi\right)^n}\int_{\mathbb{T}^n}G\left(I,\theta\right)d\theta$.\ By Fubini lemma:
    \begin{align}\nonumber
        \begin{aligned}
            &\frac{1}{lT}\int_{0}^{{lT^{-}}}\int_{\Omega}\int_{T^{n}}\bar{G}\left(I+\sum_{i=0}^{k\left(t\right)}\mathscr{I}_{\tau_{i}}\right)
            f_0\left(I,\theta\right)d\theta dI dt\\
            &=\int_{\Omega \times \mathbb{T}^n} \frac{1}{lT}\left[\int_{0}^{{lT^{-}}}\bar{G}\left(I+\sum_{i=0}^{k\left(t\right)}\mathscr{I}_{\tau_{i}}\right)dt\right]
            f_{0}\left(I,\theta\right)dId\theta\\
            &=\int_{\Omega \times \mathbb{T}^n} \frac{1}{lT}\left[\left(\int_{0}^{{T^{-}}}+\cdots+\int_{qT}^{\left(q+1\right)T^{-}}+\cdots+\int_{\left(l-1\right)T}^{lT^{-}}\right)
            \bar{G}\left(I+\sum_{i=0}^{lp-1}\mathscr{I}_{\tau_{i}}\right)dt\right]f_{0}\left(I,\theta\right)dI d\theta\\
            &=\int_{\Omega \times \mathbb{T}^n} \frac{1}{lT}\left[l \sum_{i=0}^{p-1}\left(\tau_{i+1}-\tau_{i}\right)\bar{G}\left(I+\sum_{j=0}^{i}\mathscr{I}_{\tau_{i}}\right)\right]
            f_{0}\left(I,\theta\right)dId\theta.
        \end{aligned}
    \end{align}
    Take limits:
    \begin{align}\label{162}
        \begin{aligned}
        &\displaystyle\lim_{l \rightarrow +\infty}\frac{1}{lT}\int_{0}^{{lT^{-}}} \int_{\Omega} \hat{G}\left(I+\sum_{i=0}^{k\left(t\right)}\mathscr{I}_{\tau_{i}},\vec{0}\right)\hat{f}_0\left(I,\vec{0}\right)dIdt\\
        &=\int_{\Omega \times \mathbb{T}^n}\displaystyle\lim_{l \rightarrow +\infty} \frac{1}{lT}\left[l \sum_{i=0}^{p-1}\left(\tau_{i+1}-\tau_{i}\right)\bar{G}\left(I+\sum_{j=0}^{i}\mathscr{I}_{\tau_{i}}\right)\right]f_{0}\left(I,\theta\right)dId\theta\\
        &=\int_{\Omega \times \mathbb{T}^n}\frac{1}{T}\left[ \sum_{i=0}^{p-1}\left(\tau_{i+1}-\tau_{i}\right)\bar{G}\left(I+\sum_{j=0}^{i}\mathscr{I}_{\tau_{i}}\right)\right]f_{0}\left(I,\theta\right)dId\theta\\
        &=\frac{1}{T}\sum_{i=0}^{p-1}\left(\tau_{i+1}-\tau_{i}\right)\int_{\Omega \times \mathbb{T}^n}\bar{G}\left(I+\sum_{j=0}^{i}\mathscr{I}_{\tau_{i}}\right)f_{0}\left(I,\theta\right)dId\theta\\
        &=\frac{1}{T}\sum_{i=0}^{p-1}\left(\tau_{i+1}-\tau_{i}\right)<\bar{G}\left(I+\sum_{j=0}^{i}\mathscr{I}_{\tau_{j}}\right)>_0.
        \end{aligned}
    \end{align}
    In fact, so far we have obtained the results of the theorem. In the following, we will use lemma \ref{222} to prove that the second part limit of (\ref{expect}) is zero, that is :

    \begin{align}\nonumber
        \displaystyle\lim_{l \rightarrow +\infty}\frac{1}{lT}\int_{0}^{{lT^{-}}} <G>_{t} dt = 0, \quad \vec{n} \neq \vec{0}.
    \end{align}

    By lemmas \ref{il13}, \ref{il14} and Fubini theorem, the second part of (\ref{222}) \ $\left(\vec{n} \neq \vec{0}\right)$ \ gives:
    \begin{align}\label{q2}
        \begin{aligned}
            &\frac{1}{lT}\int_{0}^{{lT^{-}}}\int_{\Omega}\sum_{\substack{\vec{n} \neq \vec{0}\\\vec{n} \in \mathbb{Z}^n}}\hat{G}\left(I+\sum_{i=0}^{k\left(t\right)}\mathscr{I}_{\tau_{i}},\vec{n}\right)\hat{f}_0\left(I,-\vec{n}\right)\\
            &\exp\left[{\sqrt{-1}<\vec{n},m\left(t\right)\Delta \theta_{T}\left(I\right)+ \sum_{i=0}^{k\left(t\right)-1}\left(\tau_{i+1}-\tau_i\right)\omega\left(I + \sum_{j=0}^{i}\mathscr{I}_{\tau_j}\right)+\left(t-\tau_{k\left(t\right)}\right)\omega\left(I+\sum_{i=0}^{k\left(t\right)}\mathscr{I}_{\tau_{i}}\right)>}\right]dIdt\\
            &=\sum_{\substack{\vec{n} \neq \vec{0}\\\vec{n} \in \mathbb{Z}^n}} \frac{1}{lT}\int_{0}^{{lT^{-}}}\int_{\Omega}\hat{G}\left(I+\sum_{i=0}^{k\left(t\right)}\mathscr{I}_{\tau_{i}},\vec{n}\right)\hat{f}_0\left(I,-\vec{n}\right)\\
            &\exp\left[{\sqrt{-1}<\vec{n},m\left(t\right)\Delta \theta_{T}\left(I\right)+ \sum_{i=0}^{k\left(t\right)-1}\left(\tau_{i+1}-\tau_i\right)\omega\left(I + \sum_{j=0}^{i}\mathscr{I}_{\tau_j}\right)+\left(t-\tau_{k\left(t\right)}\right)\omega\left(I+\sum_{i=0}^{k\left(t\right)}\mathscr{I}_{\tau_{i}}\right)>}\right]dIdt.\\
        \end{aligned}
    \end{align}

    Combine (\ref{il14}) and the dominated convergence theorem to conclude that:
    \begin{align}
        \begin{aligned}\nonumber
            &\displaystyle\lim_{l \rightarrow +\infty} \sum_{\substack{\vec{n} \neq \vec{0}\\\vec{n} \in \mathbb{Z}^n}} \frac{1}{lT}\int_{0}^{{lT^{-}}}\int_{\Omega}\hat{G}\left(I+\sum_{i=0}^{k\left(t\right)}\mathscr{I}_{\tau_{i}},\vec{n}\right)\hat{f}_0\left(I,-\vec{n}\right)\\
            &\exp\left[{\sqrt{-1}<\vec{n},m\left(t\right)\Delta \theta_{T}\left(I\right)+ \sum_{i=0}^{k\left(t\right)-1}\left(\tau_{i+1}-\tau_i\right)\omega\left(I + \sum_{j=0}^{i}\mathscr{I}_{\tau_j}\right)+\left(t-\tau_{k\left(t\right)}\right)\omega\left(I+\sum_{i=0}^{k\left(t\right)}\mathscr{I}_{\tau_{i}}\right)>}\right]dIdt\\
            &=\sum_{\substack{\vec{n} \neq \vec{0}\\\vec{n} \in \mathbb{Z}^n}}  \displaystyle\lim_{l \rightarrow +\infty} \frac{1}{lT}\int_{0}^{{lT^{-}}}\int_{\Omega}\hat{G}\left(I+\sum_{i=0}^{k\left(t\right)}\mathscr{I}_{\tau_{i}},\vec{n}\right)\hat{f}_0\left(I,-\vec{n}\right)\\
            &\exp\left[{\sqrt{-1}<\vec{n},m\left(t\right)\Delta \theta_{T}\left(I\right)+ \sum_{i=0}^{k\left(t\right)-1}\left(\tau_{i+1}-\tau_i\right)\omega\left(I + \sum_{j=0}^{i}\mathscr{I}_{\tau_j}\right)+\left(t-\tau_{k\left(t\right)}\right)\omega\left(I+\sum_{i=0}^{k\left(t\right)}\mathscr{I}_{\tau_{i}}\right)>}\right]dIdt.
        \end{aligned}
    \end{align}
    Actually, since $f_0 \in \bm{C}_c\left(\Omega \times \mathbb{T}^n\right)$, we can find a bounded and closed set $K \subset \Omega$ such that $\hat{f}_0=0$ on $ K^{c}$. Then
    \begin{align}
        \begin{aligned}\nonumber
            &\displaystyle\lim_{l \rightarrow +\infty} \frac{1}{lT}\int_{0}^{{lT^{-}}}\int_{\Omega}\hat{G}\left(I+\sum_{i=0}^{k\left(t\right)}\mathscr{I}_{\tau_{i}},\vec{n}\right)\hat{f}_0\left(I,-\vec{n}\right)\\
            &\exp\left[{\sqrt{-1}<\vec{n},m\left(t\right)\Delta \theta_{T}\left(I\right)+ \sum_{i=0}^{k-1}\left(\tau_{i+1}-\tau_i\right)\omega\left(I + \sum_{j=0}^{i}\mathscr{I}_{\tau_j}\right)+\left(t-\tau_{k\left(t\right)}\right)\omega\left(I+\sum_{i=0}^{k\left(t\right)}\mathscr{I}_{\tau_{i}}\right)>}\right]dIdt\\
            &=\displaystyle\lim_{l \rightarrow +\infty}  \frac{1}{lT}\int_{0}^{{lT^{-}}}\int_{K}\hat{G}\left(I+\sum_{i=0}^{k\left(t\right)}\mathscr{I}_{\tau_{i}},\vec{n}\right)\hat{f}_0\left(I,-\vec{n}\right)\\
            &\exp\left[{\sqrt{-1}<\vec{n},m\left(t\right)\Delta \theta_{T}\left(I\right)+ \sum_{i=0}^{k\left(t\right)-1}\left(\tau_{i+1}-\tau_i\right)\omega\left(I + \sum_{j=0}^{i}\mathscr{I}_{\tau_j}\right)+\left(t-\tau_{k\left(t\right)}\right)\omega\left(I+\sum_{i=0}^{k\left(t\right)}\mathscr{I}_{\tau_{i}}\right)>}\right]dIdt.
        \end{aligned}
    \end{align}
    Next,  we regard
    \begin{align}
        \begin{aligned}\nonumber
            &\hat{G}\left(I+\sum_{i=0}^{k\left(t\right)}\mathscr{I}_{\tau_{i}},\vec{n}\right)\hat{f}_0\left(I,-\vec{n}\right)\\
            &<\vec{n},m\left(t\right)\Delta \theta_{T}\left(I\right)+ \sum_{i=0}^{k\left(t\right)-1}\left(\tau_{i+1}-\tau_i\right)\omega\left(I + \sum_{j=0}^{i}\mathscr{I}_{\tau_j}\right)+\left(t-\tau_{k\left(t\right)}\right)\omega\left(I+\sum_{i=0}^{k\left(t\right)}\mathscr{I}_{\tau_{i}}\right)>
        \end{aligned}
    \end{align}
    as $F\left(I,t\right)$  and $\int_{0}^{t}\phi\left(I,\tau\right)d\tau$ respectively in lemma \ref{il}.

    The periodicity of the transitions and the assumptions in the theorem imply:

    $\bm{\left(1\right)}$ For any fixed $\forall t>0, \ 0=\tau_0 \leq \tau_{k} \leq t~ (\mod{T}) < \tau_{k+1} \leq \tau_{p}$, $\hat{G}\left(I+\sum_{i=0}^{k}\mathscr{I}_{\tau_{i}},\vec{n}\right)\hat{f}_0\left(I,-\vec{n}\right) \in \bm{L}^{1}\left(K\right) \subset \bm{L}^{1}\left(\Omega\right)$;

        $\bm{\left(2\right)}$ \ $\forall t>0, \ 0=\tau_0 \leq \tau_k \leq t~ (\mod{T}) < \tau_{k+1} \leq \tau_{p}$, \ $\hat{G}\left(I+\sum_{i=0}^{k}\mathscr{I}_{\tau_{i}},\vec{n}\right)\hat{f}_0\left(I,-\vec{n}\right)$ is a periodic function in variable $t$, which means  $\hat{G}\left(I+\sum_{i=0}^{k+p}\mathscr{I}_{\tau_{i}},\vec{n}\right)\hat{f}_0\left(I,-\vec{n}\right)= \hat{G}\left(I+\sum_{i=0}^{k}\mathscr{I}_{\tau_{i}},\vec{n}\right)\hat{f}_0\left(I,-\vec{n}\right)$;

        $\bm{\left(3\right)}$ $\hat{G}\left(I+\sum_{i=0}^{k}\mathscr{I}_{\tau_{i}},\vec{n}\right)\hat{f}_0\left(I,-\vec{n}\right)$ is piecewise constant with respect to variable $t$ for any fixed $I$;

        $ i.e. \  0 = \tau_0 <\tau_1 <\tau_2<\cdots<\tau_p, \tau_p - \tau_0 = T \ s.t.$
        \begin{align}
            \begin{aligned}\nonumber
                \hat{G}\left(I+\sum_{i=0}^{k\left(t\right)}\mathscr{I}_{\tau_{i}},\vec{n}\right)\hat{f}_0\left(I,-\vec{n}\right)=\hat{G}\left(I+\mathscr{I}_{\tau_{0}},\vec{n}\right)\hat{f}_0\left(I,-\vec{n}\right)  \qquad&  \tau_0 \leq t \mod{T} < \tau_1,\\
                \hat{G}\left(I+\sum_{i=0}^{k\left(t\right)}\mathscr{I}_{\tau_{i}},\vec{n}\right)\hat{f}_0\left(I,-\vec{n}\right)=\hat{G}\left(I+\mathscr{I}_{\tau_{0}}+\mathscr{I}_{\tau_{1}},\vec{n}\right)\hat{f}_0\left(I,-\vec{n}\right)  \qquad&  \tau_1 \leq t \mod{T} < \tau_2,\\
                \vdots \qquad&\\
                \hat{G}\left(I+\sum_{i=0}^{k\left(t\right)}\mathscr{I}_{\tau_{i}},\vec{n}\right)\hat{f}_0\left(I,-\vec{n}\right)=\hat{G}\left(I+\sum_{i=0}^{p-1}\mathscr{I}_{\tau_{i}},\vec{n}\right)\hat{f}_0\left(I,-\vec{n}\right)  \qquad&  \tau_{p-1} \leq t \mod{T} < \tau_p
            \end{aligned}
        \end{align}
        $\forall I \in \Omega$;

        $\bm{\left(4\right)}$ $<\vec{n},m\left(t\right)\Delta \theta_{T}\left(I\right)+ \sum_{i=0}^{k\left(t\right)-1}\left(\tau_{i+1}-\tau_i\right)\omega\left(I + \sum_{j=0}^{i}\mathscr{I}_{\tau_j}\right)+\left(t-\tau_{k\left(t\right)}\right)\omega\left(I+\sum_{i=0}^{k\left(t\right)}\mathscr{I}_{\tau_{i}}\right)> \in C^2\left(\Omega\right)$ and as $t \rightarrow lT^{-}$
        \begin{eqnarray}\nonumber
            <\vec{n},m\left(t\right)\Delta \theta_{T}\left(I\right)+ \sum_{i=0}^{k\left(t\right)-1}\left(\tau_{i+1}-\tau_i\right)\omega\left(I + \sum_{j=0}^{i}\mathscr{I}_{\tau_j}\right)+\left(t-\tau_{k\left(t\right)}\right)\omega\left(I+\sum_{i=0}^{k\left(t\right)}\mathscr{I}_{\tau_{i}}\right)>=<\vec{n},l\Delta \theta_{T} \left(I\right)>,
        \end{eqnarray}
    we can find an $\overline{\omega}_{T} \left(I\right) \in C^2\left(\Omega\right), D<\vec{n}, \overline{\omega}_{T}\left(I\right) \neq 0$ such that
        \begin{equation}\label{proofthr4}
            \displaystyle\lim_{l \rightarrow +\infty} \displaystyle\lim_{t \rightarrow lT^{-}}\frac{\frac{1}{lT} <\vec{n},m\left(t\right)\Delta \theta_{T}\left(I\right)+ \sum_{i=0}^{k\left(t\right)-1}\left(\tau_{i+1}-\tau_i\right)\omega\left(I + \sum_{j=0}^{i}\mathscr{I}_{\tau_j}\right)+\left(t-\tau_{k\left(t\right)}\right)\omega\left(I+\sum_{i=0}^{k\left(t\right)}\mathscr{I}_{\tau_{i}}\right)> - < \vec{n}, \overline{\omega}_{T}\left(I\right)>}{\frac{1}{lT}}=0,
        \end{equation}
    which means $\frac{1}{lT} <\vec{n},m\left(t\right)\Delta \theta_{T}\left(I\right)+ \sum_{i=0}^{k\left(t\right)-1}\left(\tau_{i+1}-\tau_i\right)\omega\left(I + \sum_{j=0}^{i}\mathscr{I}_{\tau_j}\right)+\left(t-\tau_{k\left(t\right)}\right)\omega\left(I+\sum_{i=0}^{k\left(t\right)}\mathscr{I}_{\tau_{i}}\right)> - < \vec{n}, \overline{\omega}_{T}\left(I\right)> = \bm{o}\left(\frac{1}{lT}\right)$.

        Because of the particularity of the integral interval ( it is an integral multiple of the period ) and the hypothesis of $\overline{\omega}_{T}\left(I\right)$, we can get (\ref{proofthr4}), but for general cases, we cannot get such an exceptional result.

    Now, using lemma \ref{il} we can conclude that:
    \begin{align}\label{168}
        \begin{aligned}
        \displaystyle\lim_{l \rightarrow +\infty} \frac{1}{lT}\int_{0}^{lT}\int_{\Omega}&\hat{G}\left(I+\sum_{i=0}^{k}\mathscr{I}_{\tau_{i}},\vec{n}\right)\hat{f}_0\left(I,-\vec{n}\right)\\
       & \exp\left[<\vec{n},m\left(t\right)\Delta \theta_{T}\left(I\right)+ \sum_{i=0}^{k\left(t\right)-1}\left(\tau_{i+1}-\tau_i\right)\omega\left(I + \sum_{j=0}^{i}\mathscr{I}_{\tau_j}\right)+\left(t-\tau_{k\left(t\right)}\right)\omega\left(I+\sum_{i=0}^{k\left(t\right)}\mathscr{I}_{\tau_{i}}\right)>\right]dIdt = 0.
    \end{aligned}
    \end{align}
    (\ref{162}) and (\ref{168}) imply : $\forall f_0 \in  \bm{C}_{c}\left(\Omega \times \mathbb{T}^{n}\right)$
    \begin{equation}\label{169}
        \displaystyle\lim_{l \rightarrow +\infty}\frac{1}{lT}\int_0^{lT}<G>_t dt = \frac{1}{T}\sum_{i=0}^{p-1}\left(\tau_{i+1}-\tau_{i}\right)<\bar{G}\left(I+\sum_{j=0}^{i}\mathscr{I}_{\tau_{j}}\right)>_0.
    \end{equation}

    Next, since $G \in \bm{C}_b\left(\tilde{\Omega} \times \mathbb{T}^{n}\right)$, we can get
    \begin{equation}\label{170}
           \sup _{\left(I,\theta \right) \in \Omega \times \mathbb{T}^{n}}|G\left(\varphi_{t}\left(I,\theta\right)\right) - \frac{1}{T}\sum_{i=0}^{p-1}\left(\tau_{i+1}-\tau_{i}\right)\bar{G}\left(I+\sum_{j=0}^{i}\mathscr{I}_{\tau_{j}}\right) | \leq M ,
    \end{equation}
    for some $M>0$.  We choose a sequence $\left\{f_{0}^{\left(k\right)}\right\}_{k=1}^{+\infty} \subset \bm{C}_{c}\left(\Omega \times \mathbb{T}^{n}\right) $ such that : \ $\forall  \ \epsilon >0, \ \exists \  n_{0}\left(\epsilon\right) >0, \ \forall \ k \geq  n_{0}\left(\epsilon\right)$,
    \begin{equation}\label{171}
        ||  f_{0} - f_{0} ^{\left(k\right)}||_{1} = \int _{\Omega \times \mathbb{T}^{n}} \big| f_{0}\left(I,\theta\right) - f_{0}^{\left(k\right)}\left(I,\theta\right) \big| dI d\theta < \frac{\epsilon}{2M}
    \end{equation}
    for any $f_{0} \in \bm{L}^{1}\left(\Omega \times \mathbb{T}^{n}\right)$, by the density of  $\bm{C}_{c}\left(\Omega \times \mathbb{T}^{n}\right)$ in $\bm{L}^{1}\left(\Omega \times \mathbb{T}^{n}\right)$. Moreover, (\ref{169}) implies: $\exists \ m_{0} \left( \epsilon \right) > 0,$ such that \ $\forall  \ m \geq m_{0}\left(\epsilon\right)$,
    \begin{align}\label{172}
        \begin{aligned}
            & \big|\frac{1}{lT}\int_{0}^{lT} \int _{\Omega \times \mathbb{T}^{n}} \left[G\left(\varphi_{t}\left(I,\theta\right)\right) - \frac{1}{T}\sum_{i=0}^{p-1}\left(\tau_{i+1}-\tau_{i}\right)\bar{G}\left(I+\sum_{j=0}^{i}\mathscr{I}_{\tau_{j}}\right)\right]f_{0}^{n_{0}\left(\epsilon\right)}dI d\theta dt \big|\\
            & = \big|\frac{1}{lT}\int_{0}^{lT} \int _{\Omega \times \mathbb{T}^{n}} G\left(\varphi_{t}\left(I,\theta\right)\right) f_{0}^{n_{0}\left(\epsilon\right)}dI d\theta dt - \frac{1}{T}\sum_{i=0}^{p-1}\left(\tau_{i+1}-\tau_{i}\right)\int_{\Omega \times \mathbb{T}^n}\bar{G}\left(I+\sum_{j=0}^{i}\mathscr{I}_{\tau_{i}}\right)f_{0}^{n_0\left(\epsilon\right)}\left(I,\theta\right)dId\theta \big| \\
            & = \big|\frac{1}{lT}\int_{0}^{lT} \int _{\Omega \times \mathbb{T}^{n}} G\left(\varphi_{t}\left(I,\theta\right)\right) f_{0}^{n_{0}\left(\epsilon\right)}dI d\theta dt - \frac{1}{T}\sum_{i=0}^{p-1}\left(\tau_{i+1}-\tau_{i}\right)\int_{\Omega \times \mathbb{T}^n}\bar{G}\left(I+\sum_{j=0}^{i}\mathscr{I}_{\tau_{i}}\right)f_{0}^{n_0\left(\epsilon\right)}\left(I,\theta\right)dId\theta \big| \\
            & < \frac{\epsilon}{2}.
        \end{aligned}
    \end{align}
    Combining (\ref{170}), (\ref{171}), (\ref{172}) gives: $\forall \ \epsilon >0, \ \exists \ m_{0}\left(\epsilon\right)>0,$ such that $\forall \ m> m_{0}\left(\epsilon\right)$,
    \begin{align}\nonumber
        \begin{aligned}
            & \big|\frac{1}{lT}\int_{0}^{lT} \int _{\Omega \times \mathbb{T}^{n}} G\left(\varphi_{t}\left(I,\theta\right)\right) f_{0}\left(I,\theta\right)dI d\theta dt -  \frac{1}{T}\sum_{i=0}^{p-1}\left(\tau_{i+1}-\tau_{i}\right)<\bar{G}\left(I+\sum_{j=0}^{i}\mathscr{I}_{\tau_{j}}\right)>_0 \big|\\
            & = \big|\frac{1}{lT}\int_{0}^{lT} \int _{\Omega \times \mathbb{T}^{n}} \left[G\left(\varphi_{t}\left(I,\theta\right)\right) - \frac{1}{T}\sum_{i=0}^{p-1}\left(\tau_{i+1}-\tau_{i}\right)\bar{G}\left(I+\sum_{j=0}^{i}\mathscr{I}_{\tau_{j}}\right)\right]f_{0}\left(I,\theta\right)dI d\theta dt \big| \\
            & \leq \big|\frac{1}{lT}\int_{0}^{lT} \int _{\Omega \times \mathbb{T}^{n}} \left[G\left(\varphi_{t}\left(I,\theta\right)\right) - \frac{1}{T}\sum_{i=0}^{p-1}\left(\tau_{i+1}-\tau_{i}\right)\bar{G}\left(I+\sum_{j=0}^{i}\mathscr{I}_{\tau_{j}}\right)\right]f_{0}^{n_{0}\left(\epsilon\right)}\left(I,\theta\right)dI d\theta dt \big|\\
            &+\big|\frac{1}{lT}\int_{0}^{lT} \int _{\Omega \times \mathbb{T}^{n}} \left[G\left(\varphi_{t}\left(I,\theta\right)\right) - \frac{1}{T}\sum_{i=0}^{p-1}\left(\tau_{i+1}-\tau_{i}\right)\bar{G}\left(I+\sum_{j=0}^{i}\mathscr{I}_{\tau_{j}}\right)\right]\left[f_{0}^{n_{0}\left(\epsilon\right)}\left(I,\theta\right)-f_{0}\left(I,\theta\right)\right]dI d\theta dt \big|\\
            & \leq \frac{\epsilon}{2}+\big| G\left(\varphi_{t}\left(I,\theta\right)\right) - \frac{1}{T}\sum_{i=0}^{p-1}\left(\tau_{i+1}-\tau_{i}\right)\bar{G}\left(I+\sum_{j=0}^{i}\mathscr{I}_{\tau_{j}}\right) \big| \frac{1}{mT}\int_{0}^{mT} \left[\int _{\Omega \times \mathbb{T}^{n}} \big| f_{0}\left(I,\theta\right) - f_{0}^{\left(n\right)}\left(I,\theta\right) \big| dI d\theta \right] dt \\
            & < \frac{\epsilon}{2}+M \big| \frac{\epsilon}{2M}  \cdot \frac{1}{mT}\int_{0}^{mT} dt   \big| = \epsilon .
        \end{aligned}
    \end{align}
    \end{proof}

    \begin{remark}
       We can get a same result with a relaxed condition : $m\left( \ \left\{ \ I\in \Omega \ | \  \nabla _{I}\overline{\omega}_{T}\left(I\right) = 0 \ \right\}  \ \right) = 0$. The proof is simple in which we can use the result in the theorem \ref{importanttheorem1}, the aditivity of regions in integrals and the substitution formula. One can see more details in \cite{mitchell2019weak}.
    \end{remark}

    According to theorem \ref{importanttheorem1}, the probability measure induced by the probability density function $f_0$ under the one-parameter flow $\varphi_{t}$ has the following weak convergence.

    \begin{theorem}\label{gailvcedu}
        Suppose that system (\ref{HI}) satisfies the conditions in theorem \ref{importanttheorem1}. Then
        \begin{align*}
            \frac{1}{lT} \int_{0} ^{lT^{-}} P_{t}  \ dt \Rightarrow \frac{1}{T}\sum_{i=0}^{p-1}\left(\tau_{i+1}-\tau_{i}\right) \overline{P}_{\tau_{i}}.
        \end{align*}
        \begin{proof}
            If  \ $G \in C_{b}\left(\Omega_{\infty} \times \mathbb{T}^n\right)$, \ then theorem \ref{importanttheorem1} implies that :
            \begin{align}\nonumber
                \begin{aligned}
                    \displaystyle\lim_{l \rightarrow +\infty}\frac{1}{lT}\int_{0}^{lT^{-}} \int_{\Omega_t \times \mathbb{T}^{n}}G\left(I,\theta\right) \ dP_t  \ dt & =\displaystyle\lim_{l \rightarrow +\infty} \frac{1}{lT}\int_{0}^{lT^{-}} \int_{\Omega_t \times \mathbb{T}^{n}}G\left(I,\theta\right) f_t\left(I,\theta\right)dId\theta  \ dt\\
                    & = \frac{1}{T}\sum_{i=0}^{p-1}\left(\tau_{i+1}-\tau_{i}\right)<\bar{G}\left(I+\sum_{j=0}^{i}\mathscr{I}_{\tau_{j}}\right)>_0, \\
                    & =  \frac{1}{T}\sum_{i=0}^{p-1}\left(\tau_{i+1}-\tau_{i}\right) \int _{\Omega \times \mathbb{T}^{n}}\bar{G}\left(I+\sum_{j=0}^{i}\mathscr{I}_{\tau_{j}}\right) f_0\left(I,\theta\right) dId\theta,  \\
                    & =  \frac{1}{T}\sum_{i=0}^{p-1}\left(\tau_{i+1}-\tau_{i}\right) \int _{\Omega \times \mathbb{T}^{n}}\frac{1}{\left(2\pi\right)^n}\int_{\mathbb{T}^n}{G}\left(I+\sum_{j=0}^{i}\mathscr{I}_{\tau_{j}},\theta\right)d\theta f_0\left(I,\theta\right) dId\theta, \\
                    & =  \frac{1}{T}\sum_{i=0}^{p-1}\left(\tau_{i+1}-\tau_{i}\right) \int _{\Omega_{\tau_i} \times \mathbb{T}^{n}}{G}\left(I,\theta\right)\left[\frac{1}{\left(2\pi\right)^n}\int_{\mathbb{T}^n}f_{\tau_{i}}\left(I,\theta\right)d\theta\right] dId\theta,\\
                    & =  \frac{1}{T}\sum_{i=0}^{p-1}\left(\tau_{i+1}-\tau_{i}\right) \int _{\Omega_{\tau_i} \times \mathbb{T}^{n}}{G}\left(I,\theta\right)d\overline{P}_{\tau_{i}},
                \end{aligned}
            \end{align}
            which proves the theorem.
        \end{proof}
    \end{theorem}

    \section{Almost periodic case}
    Next, let's extend the result to the case in which the transitions are almost periodic.

    We call the transitions almost periodic if 
    \begin{equation}
        \begin{aligned}
          \hat{I}_{I_0}\left(n\right) = I_0+\sum_{i=0}^{n}\mathscr{I}_{t_i}
        \end{aligned}   
    \end{equation}
   is a almost periodic sequence with respect to $n$, for any $I_0 \in \Omega$, which means to any $\epsilon > 0$ there corresponds
   an integer $N\left(\epsilon\right)$, such that among any $N$ consecutive integers there exists an integer $p$ with the property \cite{1989Almost}
   \begin{equation}
    \begin{aligned}
        |\hat{I}_{I_0}\left(n+p\right)-\hat{I}_{I_0}\left(n\right)| < \epsilon
    \end{aligned}
   \end{equation}
   Without loss of  generality, in this section, we consider the ensemble problem of the following system
   \begin{equation}\label{apc}
    \left\{
    \begin{aligned}
        \ I & =  \hat{I}_{I_0}\left([t]\right) \\
        \ \theta & =  \theta_0 + \sum_{i=0}^{[t]-1}\omega\left(\hat{I}_{I_0}\left(i\right)\right) + \omega\left(\hat{I}_{I_0}\left([t]\right)\right)\left(t-[t]\right)
    \end{aligned}
    \right.
\end{equation}
where $\hat{I}_{I_0}\left(n\right) = I_0 + \sum_{i=0}^{n}\mathscr{I}_{i}$.

Since almost periodic sequences and the set $\Omega$ are bounded, $\Omega_{\infty}= \left\{ \ I \ | \ I = \hat{I}_{I_0}\left(n\right) , I_0 \in \Omega, n \in \mathbb{N} \ \right\}$ is also bounded.
Before giving the main results of this section, we need to make some modifications to the conditions in Theorem \ref{importanttheorem1}.
\begin{equation}\label{conditionc}
    \begin{aligned}
        \forall \ N \in \mathbb{N}, \ \ \overline{\omega}_{N} \left(I\right) = \frac{1}{N} \sum_{i=0}^{N-1}\omega \left(\hat{I}_{I}\left(n\right)\right) \in C^2\left(\Omega\right) \ and \ have \ no \ critical \ points.
    \end{aligned}
\end{equation}
\begin{theorem}
    In system \ref{apc}, suppose that $\omega \left(I\right) \in C^2\left(\Omega_{\infty}\right)$ and satisfies \ref{conditionc}. The initial condition is described by a probability density function $f_0 \in \bm{L}^{1}\left(\Omega \times \mathbb{T}^{n}\right)$.
    For any continuous and bounded function $G\in \bm{C}_b\left(\Omega_{\infty}\times \mathbb{T}^{n}\right)$, (\ref{222}) has limit:
    $$\displaystyle\lim_{N \rightarrow +\infty}\frac{1}{N}\int_{0}^{N} <G>_t dt =  \displaystyle\lim_{N \rightarrow +\infty}\frac{1}{N} \sum_{i=0}^{N-1}<G\left(\hat{I}_{I}\left(i\right)\right)>_0$$
\end{theorem}
\begin{proof}
    \begin{equation}
        \begin{aligned}
        <G>_{t} &= \int_{\Omega \times \mathbb{T}^{n}}  G\left(\hat{I}_{I}\left([t]\right),\theta+\sum_{i=0}^{[t]-1}\omega\left(\hat{I}_{I}\left(i\right)\right)+\omega\left(\hat{I}_{I}\left([t]\right)\right)\left(t-[t]\right)\right)f_t\left(I,\theta\right)dId\theta \\
        &=\left(2\pi\right)^{n} \int_{\Omega} \sum_{\vec{n} \in \mathbb{Z}^{n}}\hat{G}\left(\hat{I}_{I}\left([t]\right),\vec{n}\right)\hat{f}_{0}\left(I,-\vec{n}\right)\exp\left(\sqrt{-1}<\vec{n},\sum_{i=0}^{[t]-1}\omega\left(\hat{I}_{I}\left(i\right)\right)+\omega\left(\hat{I}_{I}\left([t]\right)\right)\left(t-[t]\right)>\right)dI
        \end{aligned}
    \end{equation}
    For sufficiently large $k$, taking $\epsilon = \frac{1}{k^4}$, we can get $p\in \mathbb{N}$ such that $|\hat{I}_{I}\left(n+p\right)-\hat{I}_{I}\left(n\right)| < \epsilon, \ \forall I \in \Omega$.

    To simplify the following symbols, we set 
    \begin{equation}
        \begin{aligned}
        \hat{G}\left(I+k\sum_{i=0}^{p-1}\mathscr{I}_{i},\vec{n}\right) &= \hat{G}_{k,p-1}^{\vec{n}}\left(I\right)\\
        \hat{G}\left(I+\sum_{i=0}^{kp-1}\mathscr{I}_{i},\vec{n}\right) &= \hat{G}_{kp-1}^{\vec{n}}\left(I\right) \\
        \exp[\sqrt{-1}<\vec{n},k\sum_{i=0}^{p-1}\omega\left(I+\sum_{j=0}^{i}\mathscr{I}_{j}\right)>] &= e_{k,p-1}^{\vec{n}}\left(I\right) \\
        \exp[\sqrt{-1}<\vec{n},\sum_{i=0}^{kp-1}\omega\left(I+\sum_{j=0}^{i}\mathscr{I}_{j}\right)>] &= e_{kp-1}^{\vec{n}}\left(I\right) 
        \end{aligned}
    \end{equation}
    Firstly, let's consider the situation of $\vec{n} \neq \vec{0}$.

    As $k \rightarrow \infty$, by lemma \ref{il},
    \begin{equation}
        \begin{aligned}
            & \int_{\Omega}\hat{G}_{k,p-1}^{\vec{n}}\left(I\right)\hat{f}_{0}\left(I,-\vec{n}\right)e_{k,p-1}^{\vec{n}}\left(I\right)dI \rightarrow 0 \  , \\
            & \int_{\Omega}\left(\hat{G}_{k,p-1}^{\vec{n}}-\hat{G}_{kp-1}^{\vec{n}}\right)\left(I\right)\hat{f}_{0}\left(I,-\vec{n}\right)e_{k,p-1}^{\vec{n}}\left(I\right)dI \rightarrow 0 \ ,
        \end{aligned}
    \end{equation}
    in addtion,
    \begin{equation}
        \begin{aligned}
            |e_{k,p-1}^{\vec{n}}\left(I\right)-e_{kp-1}^{\vec{n}}\left(I\right)| & \leq |\sin <\frac{\vec{n}}{2},k\sum_{i=0}^{p-1}\omega\left(\hat{I}_{I}\left(i\right)\right)-\sum_{i=0}^{kp-1}\omega\left(\hat{I}_{I}\left(i\right)\right)>| \\
            & \leq \frac{|\vec{n}||\omega|}{2}|0+p\epsilon+2p\epsilon+\cdots+(k-1)p\epsilon| \\
            & \leq \frac{|\vec{n}||\omega|p\left(k-1\right)}{4k^3} \rightarrow 0
        \end{aligned}
    \end{equation}
Hence,
\begin{equation}
    \begin{aligned}
        & |\int_{\Omega}\hat{G}_{k,p-1}^{\vec{n}}\left(I\right)\hat{f}_{0}\left(I,-\vec{n}\right)e_{k,p-1}^{\vec{n}}\left(I\right)dI-\int_{\Omega}\hat{G}_{kp-1}^{\vec{n}}\hat{f}_{0}\left(I,-\vec{n}\right)e_{kp-1}^{\vec{n}}\left(I\right)dI| \\
        & \leq |\int_{\Omega}\left(\hat{G}_{k,p-1}^{\vec{n}}-\hat{G}_{kp-1}^{\vec{n}}\right)\left(I\right)\hat{f}_{0}\left(I,-\vec{n}\right)e_{k,p-1}^{\vec{n}}\left(I\right)dI| + |\int_{\Omega}\hat{G}_{kp-1}^{\vec{n}}\left(I\right)\hat{f}_{0}\left(I,-\vec{n}\right)\left(e_{k,p-1}^{\vec{n}}-e_{kp-1}^{\vec{n}}\right)\left(I\right)dI|\\
        & \leq |\int_{\Omega}\left(\hat{G}_{k,p-1}^{\vec{n}}-\hat{G}_{kp-1}^{\vec{n}}\right)\left(I\right)\hat{f}_{0}\left(I,-\vec{n}\right)e_{k,p-1}^{\vec{n}}\left(I\right)dI| + |\Omega||G|\frac{|\vec{n}||\omega|p\left(k-1\right)}{4k^3} \rightarrow 0
    \end{aligned}
\end{equation}
which means $\forall \vec{n} \in \mathbb{Z}^{n}$, $\vec{n} \neq \vec{0}$,
\begin{equation}
        \displaystyle\lim_{N \rightarrow +\infty}\frac{1}{N}\int_{0}^{N} \left(2\pi\right)^{n}\int_{\Omega}\sum_{\substack{\vec{n} \neq \vec{0}\\\vec{n} \in \mathbb{Z}^n}}\hat{G}\left(\hat{I}_{I}\left([t]\right),\vec{n}\right)\hat{f}_{0}\left(I,-\vec{n}\right)\exp\left(\sqrt{-1}<\vec{n},\sum_{i=0}^{[t]-1}\omega\left(\hat{I}_{I}\left(i\right)\right)+\omega\left(\hat{I}_{I}\left([t]\right)\right)\left(t-[t]\right)>\right)dIdt=0.
\end{equation}
So,
\begin{equation}
    \begin{aligned}
        \displaystyle\lim_{N \rightarrow +\infty}\frac{1}{N}\int_{0}^{N} <G>_t dt &=  \displaystyle\lim_{N \rightarrow +\infty}\frac{1}{N}\int_{0}^{N} \left(2\pi\right)^{n}\int_{\Omega}\hat{G}\left(\hat{I}_{I}\left([t]\right),\vec{0}\right)\hat{f}_{0}\left(I,-\vec{0}\right)dIdt=0.\\
        &=\displaystyle\lim_{N \rightarrow +\infty}\frac{1}{N} \sum_{i=0}^{N-1}<G\left(\hat{I}_{I}\left(i\right)\right)>_0
    \end{aligned}
\end{equation}
\end{proof}
\begin{remark}
    Like Theorem \ref{gailvcedu}, we can also define probability measures and obtain their weak convergence:
    \begin{equation}
        \displaystyle\lim_{N \rightarrow +\infty}\frac{1}{N}\int_{0}^{N}P_{t}dt \Rightarrow \displaystyle\lim_{N \rightarrow +\infty}\frac{1}{N}\sum_{i=0}^{N-1}\bar{P}_{i}
    \end{equation}
\end{remark}

    {\bf Data Availability} No data was used in the research described in this paper.

    \section*{Acknowledgment}
                The work of Li, Y. is partially  supported by National Basic Research Program of China (No.
                2013CB834100),  Science and Technology Developing Plan of Jilin Province
                (No. 20190201302JC), 

    \section*{Declarations}
                {\bf Conflict of interest} The authors declare that they have no conflict of interest.

    \section*{Appendix A}\label{A1}
    We give the detailed derivation of one-parameter flow and its inverse map here.
    \subsection{The one-parameter flow and its inverse map in $\tau_1 \leq t < \tau_{2}$}\label{4.1}

    System (\ref{HI}) defines a corresponding  flow on $\Omega \times \mathbb{T}^n$ during $\tau_{1} <  t  \leq \tau_{2}$ given by the one-parameter family of maps:
    \begin{equation*}
        \begin{aligned}
            \varphi _t : \quad \Omega \times \mathbb{T}^n & \rightarrow \Omega_{\tau_{1}} \times \mathbb{T}^n \\
            \left(I,\theta\right)  &\rightarrow \varphi_t\left(I,\theta\right)=\left(I+\mathscr{I}_{\tau_{1}},\theta+\left(\tau_{1}-\tau_{0}\right)\omega\left(I\right)+\left(t- \tau_{1}\right)\omega\left(I+\mathscr{I}_{\tau_{1}}\right)\right),
            \end{aligned}
    \end{equation*}
    where $\Omega_{\tau_{1}}=\left\{I : I=I_0+\mathscr{I}_{\tau_{1}} , I_0 \in \Omega \right\}$. One can simply see that the $\varphi_t$ is an injection. What's more

    \begin{equation*}
        D\varphi_t=
        \begin{bmatrix}
            \bm{1_n}&0\\
            \left(\tau_{1}-\tau_{0}\right)\omega_{I}\left(I\right)+\left(t- \tau_{1}\right)\omega_{I}\left(I+\mathscr{I_{\tau_{1}}}\right)&\bm{1_n}
        \end{bmatrix},
        \qquad det\left(D\varphi_t\right)=1.
    \end{equation*}
    The inverse map of $\varphi_t$ on $\Omega_{\tau_{1}} \times \mathbb{T}^n$ is defined as :
    \begin{equation*}
        \begin{aligned}
            \varphi _t^{-1} : \quad \Omega_{\tau_{1}} \times \mathbb{T}^n & \rightarrow \Omega \times \mathbb{T}^n \\
            \left(I,\theta\right)  &\rightarrow\left(\varphi _t^{-1}\left(I,\theta\right)_{\left(I\right)},\varphi _t^{-1}\left(I,\theta\right)_{\left( \theta \right)}\right).
            \end{aligned}
    \end{equation*}
    On the one hand,
    \begin{equation*}
        \varphi_t \circ \varphi_t^{-1}\left(I,\theta\right)=\left(I,\theta\right),
    \end{equation*}
    on the other hand,
    \begin{equation*}
        \begin{aligned}
    \varphi_t &\left( \varphi _t^{-1}\left(I,\theta\right)_{\left(I\right)},\varphi _t^{-1}\left(I,\theta\right)_{\left( \theta \right)}\right)\\
              &=\left(\varphi _t^{-1}\left(I,\theta\right)_{\left(I\right)}+ \mathscr{I}_{\tau_{1}},\varphi _t^{-1}\left(I,\theta\right)_{\left( \theta \right)}+\left(\tau_{1}-\tau_{0}\right)\omega\left(\varphi _t^{-1}\left(I,\theta\right)_{\left(I\right)}\right)+\left(t- \tau_{1}\right)\omega\left(\varphi _t^{-1}\left(I,\theta\right)_{\left(I\right)}+\mathscr{I}_{\tau_{1}}\right)\right).
        \end{aligned}
    \end{equation*}
    Then
    \begin{equation*}
        \left\{
        \begin{aligned}
            I &= \varphi _t^{-1}\left(I,\theta\right)_{\left(I\right)}+ \mathscr{I}_{\tau_{1}} ,\\
        \theta &=\varphi _t^{-1}\left(I,\theta\right)_{\left( \theta \right)}+\left(\tau_{1}-\tau_{0}\right)\omega\left(\varphi _t^{-1}\left(I,\theta\right)_{\left(I\right)}\right)+\left(t- \tau_{1}\right)\omega\left(\varphi _t^{-1}\left(I,\theta\right)_{\left(I\right)}+\mathscr{I}_{\tau_{1}}\right),
        \end{aligned}
        \right.
    \end{equation*}
    which means:
    \begin{equation*}
        \left\{
        \begin{aligned}
            \varphi _t^{-1}\left(I,\theta\right)_{\left(I\right)} &=I-\mathscr{I}_{\tau_{1}} ,\\
            \varphi _t^{-1}\left(I,\theta\right)_{\left( \theta \right)} &=\theta -\left(\tau_{1}-\tau_{0}\right)\omega\left(I-\mathscr{I}_{\tau_{1}}\right)-\left(t- \tau_{1}\right)\omega\left(I\right).
        \end{aligned}
        \right.
    \end{equation*}
    Now we get the inverse map $\varphi_t^{-1}:$
    \begin{equation*}
        \varphi_t^{-1}(I,\theta)=\left(I-\mathscr{I}_{\tau_{1}} ,\theta -\left(\tau_{1}-\tau_{0}\right)\omega\left(I-\mathscr{I}_{\tau_{1}}\right)-\left(t- \tau_{1}\right)\omega\left(I\right)\right),\qquad \forall \left(I,\theta\right) \in \Omega_{t}\times \mathbb{T}^n,
    \end{equation*}
    \begin{equation*}
        D\varphi_t^{-1}=
        \begin{bmatrix}
            \bm{1_n}&0\\
            -\left(\tau_{1}-\tau_{0}\right)\omega_{I}\left(I-\mathscr{I}_{\tau_{1}}\right)-\left(t- \tau_{1}\right)\omega_{I}\left(I\right)&\bm{1_n}
        \end{bmatrix},
        \qquad det\left(D\varphi_t^{-1}\right)=1.
    \end{equation*}
    \subsection{The one-parameter flow and its inverse map in $\tau_0 < \tau_k \leq t < \tau_{k+1} \leq \tau_{p}$}\label{sec:11}
    Like the subsection \ref{4.1}, system (\ref{HI}) defines a corresponding  flow on $\Omega \times \mathbb{T}^n$ in $\tau_0 < \tau_k \leq t < \tau_{k+1} \leq \tau_{p}$ :
    \begin{equation*}
        \begin{aligned}
            \varphi _t : \quad \Omega \times \mathbb{T}^n & \rightarrow \Omega_{\tau_{k}} \times \mathbb{T}^n \\
            \left(I,\theta\right)  &\rightarrow \varphi_t\left(I,\theta\right)=\left(I+\sum_{i=0}^{k}\mathscr{I}_{\tau_{i}},\theta+\sum_{i=0}^{k-1}\left(\tau_{i+1}-\tau_i\right)\omega\left(I + \sum_{j=0}^{i}\mathscr{I}_{\tau_j}\right)+\left(t-\tau_{k}\right)\omega\left(I+\sum_{i=0}^{k}\mathscr{I}_{\tau_{i}}\right)\right),
            \end{aligned}
    \end{equation*}
    where $\Omega_{\tau_{k}}=\left\{I : I=I_0+\sum_{i=0}^{k}\mathscr{I}_{\tau_{i}} , I_0 \in \Omega \right\}$. Note
    \begin{equation*}
        D\varphi_t=
        \begin{bmatrix}
            \bm{1_n}&0\\
             \sum_{i=0}^{k-1}\left(\tau_{i+1}-\tau_i\right)\omega_{I}\left(I+\sum_{j=0}^{i}\mathscr{I}_{\tau_{j}}\right)+\left(t- \tau_{k}\right)\omega_{I}\left(I+\sum_{i=0}^{k}\mathscr{I}_{\tau_{i}}\right)&\bm{1_n}
        \end{bmatrix},
    \end{equation*}
    \begin{equation*}
        det\left(D\varphi_t\right)=1.
    \end{equation*}
The inverse map of $\varphi_t$ on $\Omega_{\tau_{k}} \times \mathbb{T}^n$ is defined as:
    \begin{equation*}
        \begin{aligned}
            \varphi _t^{-1} : \quad \Omega_{\tau_{k}} \times \mathbb{T}^n & \rightarrow \Omega \times \mathbb{T}^n \\
            \left(I,\theta\right)  &\rightarrow\left(\varphi _t^{-1}\left(I,\theta\right)_{\left(I\right)},\varphi _t^{-1}\left(I,\theta\right)_{\left( \theta \right)}\right).
            \end{aligned}
    \end{equation*}
    On the one hand,
    \begin{equation*}
        \varphi_t \circ \varphi_t^{-1}\left(I,\theta\right)=\left(I,\theta\right),
    \end{equation*}
    on the other hand,
    \begin{equation*}
        \begin{aligned}
    &\varphi_t \left( \varphi _t^{-1}\left(I,\theta\right)_{\left(I\right)},\varphi _t^{-1}\left(I,\theta\right)_{\left( \theta \right)}\right)\\
              &=\left(\varphi _t^{-1}\left(I,\theta\right)_{\left(I\right)}+\sum_{i=0}^{k}\mathscr{I}_{\tau_{i}},\varphi _t^{-1}\left(I,\theta\right)_{\left( \theta \right)}+\sum_{i=0}^{k-1}\left(\tau_{i+1}-\tau_i\right)\omega\left(\varphi _t^{-1}\left(I,\theta\right)_{\left(I\right)} + \sum_{j=0}^{i}\mathscr{I}_{\tau_j}\right)+\left(t-\tau_{k}\right)\omega\left(\varphi _t^{-1}\left(I,\theta\right)_{\left(I\right)}+\sum_{i=0}^{k}\mathscr{I}_{\tau_{i}}\right)\right).
        \end{aligned}
    \end{equation*}
    Then
    \begin{equation*}
        \left\{
        \begin{aligned}
            I &= \varphi _t^{-1}\left(I,\theta\right)_{\left(I\right)}+\sum_{i=0}^{k}\mathscr{I}_{\tau_{i}} ,\\
        \theta &=\varphi _t^{-1}\left(I,\theta\right)_{\left( \theta \right)}+\sum_{i=0}^{k-1}\left(\tau_{i+1}-\tau_i\right)\omega\left(\varphi _t^{-1}\left(I,\theta\right)_{\left(I\right)} + \sum_{j=0}^{i}\mathscr{I}_{\tau_j}\right)+\left(t-\tau_{k}\right)\omega\left(\varphi _t^{-1}\left(I,\theta\right)_{\left(I\right)}+\sum_{i=0}^{k}\mathscr{I}_{\tau_{i}}\right),
        \end{aligned}
        \right.
    \end{equation*}
    which means:
    \begin{equation*}
        \left\{
        \begin{aligned}
            \varphi _t^{-1}\left(I,\theta\right)_{\left(I\right)} &=I-\sum_{i=0}^{k}\mathscr{I}_{\tau_{i}} ,\\
            \varphi _t^{-1}\left(I,\theta\right)_{\left( \theta \right)} &=\theta -\sum_{i=0}^{k-1}\left(\tau_{i+1}-\tau_i\right)\omega\left(I-\sum_{j=i+1}^{k}\mathscr{I}_{\tau_{j}}\right)-\left(t- \tau_{k}\right)\omega\left(I\right).
        \end{aligned}
        \right.
    \end{equation*}
    Now we get the inverse map $\varphi_t^{-1}$ on $\Omega_{t}\times \mathbb{T}^n$:
    \begin{equation*}
        \varphi_t^{-1}(I,\theta)=\left(I-\sum_{i=0}^{k}\mathscr{I}_{\tau_{i}} ,\theta -\sum_{i=0}^{k-1}\left(\tau_{i+1}-\tau_i\right)\omega\left(I-\sum_{j=i+1}^{k}\mathscr{I}_{\tau_{j}}\right)-\left(t- \tau_{k}\right)\omega\left(I\right)\right),
    \end{equation*}
    and
    \begin{equation*}
        D\varphi_t^{-1}=
        \begin{bmatrix}
            \bm{1_n}&0\\
             -\sum_{i=0}^{k-1}\left(\tau_{i+1}-\tau_i\right)\omega_{I}\left(I-\sum_{j=i+1}^{k}\mathscr{I}_{\tau_{j}}\right)-\left(t- \tau_{k}\right)\omega_{I}\left(I\right)&\bm{1_n}
        \end{bmatrix},
    \end{equation*}
    \begin{equation*}
        det\left(D\varphi_t^{-1}\right)=1.
    \end{equation*}
    where the $\bm{1_n}$ denotes the $n$-order identity matrix.

    \section*{Appendix B}
    \begin{lemma}[Riemann-Lebesgue lemma](see \cite{stein1993harmonic})

        Let $\Omega \subset \mathbb{R}^{n}$ be open. Suppose that $a \in L^{1} \left(\Omega\right),$ and that $\phi \in C^{2}\left(\Omega\right)$ is real-valued with $\nabla \phi \neq 0$.
        Then for $\lambda \in \mathbb{R},$
        \begin{equation}
            I\left(\lambda\right) = \int_{\Omega} a\left(x\right) \exp\left[\sqrt{-1}\lambda \phi\left(x\right)\right]dx \rightarrow 0  \qquad  as \qquad |\lambda|\rightarrow \infty.\nonumber
        \end{equation}
    \end{lemma}
    \begin{theorem}[Parseval's theorem](see \cite{danese1965advanced})

        Suppose $f\left(x\right)$ is a square-integrable function over $[-\pi,\pi]$ (i.e. $f\left(x\right)$ and $f^{2}\left(x\right)$ are integrable on that interval), with the Fourier series
        \begin{equation}
            f \left(x\right) \simeq \frac{a_0}{2} + \sum_{n=1}^{\infty}\left(a_n\cos \left(nx\right)+b_n\sin\left(nx\right)\right).\nonumber
        \end{equation}
        Then
        \begin{equation}
            \frac{1}{\pi}\int_{-\pi}^{\pi}f^{2}\left(x\right)dx = \frac{a_0^2}{2} + \sum_{i=1}^{\infty}\left(a_{i}^{2}+b_{i}^{2}\right).\nonumber
        \end{equation}
    \end{theorem}
    \begin{theorem}[Fubini theorem](see \cite{royden1988real})\label{FT}

        Let $\left(X,\mathcal{A},\mu\right)$ and $\left(Y,\mathcal{B},\nu\right)$ be two measure spaces and $\nu$ be complete. Let $f$ be integrable over $X \times Y$
        with respect to the product measure $\mu \times \nu$. Then for almost all $x\in X$, the $x-$section of $f$, $f\left(x,\cdot\right)$, is integrable over $Y$ with respect to $\nu$ and
        \begin{equation}
            \int_{X \times Y}f d\left(\mu \times \nu\right)=\int_{X}\left[\int_{Y}f\left(x,y\right)d\nu\left(y\right)\right]d\mu\left(x\right).\nonumber
        \end{equation}
        \end{theorem}
        \begin{theorem}[Lebesgue's dominated convergence theorem](see \cite{book91750210})

            Let $f_{i}$ be a sequence of complex-valued measurable functions on a measure space $(S, \Sigma, \mu)$. Suppose that the sequence converges pointwise to a function $f$ and is dominated by some integrable function $g$ in the sense that
            \begin{equation}
                |f_{i}\left(x\right)| \leq g\left(x\right)\nonumber
            \end{equation}
            for all numbers $i$ in the index set of the sequence and all points $x \in S$. Then $f$ is integrable (in the Lebesgue sense) and
            \begin{equation}
                \lim_{i \rightarrow \infty}\int_{S}|f_{i}\left(x\right)-f\left(x\right)| d\mu = 0,\nonumber
            \end{equation}
            which also implies
            \begin{equation}
                \lim_{i \rightarrow \infty}\int_{S}f_{i}\left(x\right) d\mu =\int_{S}f\left(x\right) d\mu.\nonumber
            \end{equation}
        \end{theorem}

    \newpage

\end{document}